\documentclass[11pt]{aims}
\usepackage{amsmath, amssymb, mathrsfs}
  \usepackage{paralist}
  \usepackage{graphics} 
  \usepackage{epsfig} 
\usepackage{graphicx}  \usepackage{epstopdf}
 \usepackage[colorlinks=true]{hyperref}
 \hypersetup{urlcolor=blue, citecolor=red}
 \usepackage[toc,page]{appendix}
\usepackage{chngcntr}

\usepackage{listings}
\usepackage{color} 
\definecolor{mygreen}{RGB}{28,172,0} 
\definecolor{mylilas}{RGB}{170,55,241}

\usepackage{titlesec}
\titleformat{\section}{\Large\bfseries}{\thesection.}{4pt}{}
\titleformat{\subsection}{\large\bfseries}{\thesection.\arabic{subsection}.}{4pt}{}
\titleformat{\subsubsection}{\bfseries}{\thesection.\arabic{subsection}.\arabic{subsubsection}.}{4pt}{}
\titleformat*{\paragraph}{\bfseries}
\titleformat*{\subparagraph}{\bfseries}
\usepackage[margin=1.2in]{geometry}

  \def\currentyear{}
   \def\currentmonth{}
    \def\ppages{}
     \def\DOI{}

\newtheorem{theorem}{Theorem}[section]

\newtheorem{lemma}[theorem]{Lemma}
\newtheorem{proposition}[theorem]{Proposition}

\theoremstyle{definition}
\newtheorem{definition}[theorem]{Definition}
\newtheorem{remark}[theorem]{Remark}

\newcommand{\vep}{\varepsilon}

\newcommand{\yj}{\langle y \rangle}

\newcommand{\Rb}{\mathbb{R}}

\newcommand{\Kc}{\mathcal{K}}
\newcommand{\Vc}{\mathcal{V}}
\newcommand{\Sc}{\mathcal{S}}
\newcommand{\Ec}{\mathcal{E}}
\newcommand{\Oc}{\mathcal{O}}

\newcommand{\Cc}{\mathcal{C}}

\newcommand{\Dc}{\mathcal{D}}

\newcommand{\pa}{\partial}

\newcommand{\out}{\textup{ex}}
\newcommand{\Ls}{\mathscr{L}}

\newcommand{\Hs}{\mathscr{H}}
\newcommand{\As}{\mathscr{A}}

\numberwithin{equation}{section}

\usepackage{xcolor}
\usepackage{setspace} 
\usepackage{enumerate}
\usepackage{color, comment}

\title[Type I-Log singularity in the 3D and 4D Keller-segel system] 
      {Construction of type I-Log blowup for the Keller-Segel system in dimensions $3$ and $4$.}

\author[V. T. Nguyen, N. Nouaili, H. Zaag]{}

\email[N. Nouaili]{nouaili@ceremade.dauphine.fr}
 \email[V. T. Nguyen]{vtnguyen@ntu.edu.tw}
 \email[H. Zaag]{Hatem.Zaag@math.cnrs.fr}

\thanks{
\today}
\begin{document}
\maketitle

\maketitle
	
	
	\centerline{Van Tien Nguyen$^{(1)}$, Nejla Nouaili$^{(2)}$ and Hatem Zaag$^{(3)}$} 
	\medskip
	{\footnotesize
	\centerline{ $^{(1)}$ Department of Mathematics,
Institute of Applied Mathematical Sciences,
National Taiwan University. }
 \centerline{ $^{(2)}$ CEREMADE, Universit\'e Paris Dauphine, Paris Sciences et Lettres, France }
\centerline{ $^{(3)}$Universit\'e Sorbonne Paris Nord,
			LAGA, CNRS (UMR 7539), F-93430, Villetaneuse, France.}
	}

	\bigskip
	\begin{center}\thanks{\today}\end{center}

\begin{abstract} We construct finite time blowup solutions to the parabolic-elliptic Keller-Segel system
$$\pa_t u = \Delta u - \nabla \cdot (u \nabla \Kc_u), \quad -\Delta \Kc_u = u \quad \textup{in}\;\; \Rb^d,\; d = 3,4,$$
and derive the final blowup profile
$$ u(r,T) \sim c_d \frac{|\log r|^\frac{d-2}{d}}{r^2} \quad \textup{as}\;\; r \to 0, \;\; c_d > 0.$$
To our knowledge this provides a new blowup solution for the Keller-Segel system, rigorously answering a question by Brenner \textit{et al} in \cite{BCKSVnon99}. 
\end{abstract}

\bigskip

\section{Introduction.}
We are interested in the existence of blowup solutions to the Keller-Segel system 
\begin{equation}\label{sys:KS}
\left\{ \begin{array}{rl}
\pa_t u = &\Delta u - \nabla \Kc_u \cdot\nabla u + u^2, \\
0 = &\Delta \Kc_u + u, 
\end{array} \right. \quad x \in \Rb^d,
\end{equation}
where $u(t): x \in \Rb^d \to \Rb$ subject to initial data $u(0) = u_0$. The system \eqref{sys:KS} appears in many biological and astrophysical contexts. Here, $u(x,t)$ stands for the density of particles or cells and $\Kc_u$ is a self-interaction potential. In the two-dimensional case, it is used to model the so-called \textit{chemotaxis} phenomena in biology first introduced by Patlak \cite{PATbmb53} and Keller-Segel in \cite{KSjtb70} (see also \cite{KSjtb71a} \cite{KSjtb71b} for a derivation of a general model). In higher dimensional cases, the system \eqref{sys:KS} appears as a simplified model for self-gravitating matter in stellar dynamics, see for example  \cite{Warma92}, \cite{Wjam92},  \cite{SCPRSN02}, \cite{CSPRE11}  and \cite{DLMN13}. We refer to the papers \cite{Hjdmv03}, \cite{Cepj08} where the authors give a nice survey of mathematical derivation of \eqref{sys:KS} and related models. It is worth mentioning that the system \eqref{sys:KS} is a special case belonging to a much wider class of nonlocal aggregation equations including those with degenerate diffusion read as
\begin{equation}\label{eq:generalequation}
\pa_t u = \Delta A(u) + \nabla \cdot(B(u) \nabla \Kc \ast u), \quad  (x,t) \in \Rb^d \times [0, \infty), 
\end{equation}
where $A(u)$ and $B(u)$ can be nonlinear functions and $\Kc$ is an arbitrary local integrable function. In \eqref{sys:KS}, we have $A(u) = B(u) = u$ and $\Kc$ is the Newtonian potential. Besides covering a wide range of applications, equation \eqref{eq:generalequation} posses an interesting mathematical phenomenon already observed in \eqref{sys:KS}: the competition between the diffusion and the nonlocal aggregation for which we may have global existence or finite time singularity of solutions.

\subsection{What is known}
The local Cauchy problem for \eqref{sys:KS} can be solved in $L^\infty(\Rb^d)$ for the class of radially symmetric solutions, see for example Karch-Suzuki \cite{KSAM11}, Souplet-Winkler \cite{SWCMP19}, Winkler \cite{Wom23}. For the study of Cauchy problem for the general model \eqref{eq:generalequation}, we refer the works of Masmoudi-Bedrossian \cite{BMARMA14} (see also \cite{BRBN11}, \cite{BRDCDS14}, \cite{BCPDE15}), Biler-Karch-Pilarczyk \cite{BKPJDE19} and references therein.  Due to the singularity of the Newtonian potential at the origin, the existence of finite-time singularity (blowup) solutions and their mechanism are delicate, and this is precisely our focus in this paper.\\ 

\paragraph{Singularities in the classical nonlinear heat equation:} Neglecting the drift term in \eqref{sys:KS} gives the classical semilinear heat equation 
\begin{equation} \label{eq:NLH}
\pa_t u = \Delta u + |u|^{p-1}u \quad \textup{with} \;\; p = 2,
\end{equation}
where the existence of global-in-time or blowup solutions has received lots of attention in the last five decades, starting from the work of Fujita \cite{FUJsut66}. We refer to the book of Quittner-Souplet \cite{QSbook07} as a nice source of references on this subject.  Up to now, a fairly completed picture on the singularity formation for \eqref{eq:NLH} has been established, especially for the existence with a detailed description of the mechanism near the singularity. A very first  result  in this direction based on a numerical rescaling algorithm by Berger-Kohn \cite{BKcpam88} suggested that a stable (generic) blowup solution to \eqref{eq:NLH} is given by 
\begin{equation}\label{eq:profileNLH}
u(x,t) \sim \frac{1}{T- t}P\left( \frac{x}{\sqrt{(T-t)\log(T-t)}}\right), \quad P(\xi) = \frac{1}{1 + |\xi^2|/8}. 
\end{equation}
A rigorous construction was done by  Bricmont-Kupiainen \cite{BKnon94}, Merle-Zaag \cite{MZdm97}, and a refined description by Nguyen-Zaag \cite{NZens17}. A completed classification of all Type I blowup behaviors was established by Fillipas-Kohn \cite{FKcpam92}, Filippas-Liu \cite{FLaihn93}, Vel\'azquez \cite{VELcpde92}, Herrero-Vel\'azquez \cite{HVcpde92} where we see unstable blowup profiles. Here, a blowup solution to \eqref{eq:NLH} is of Type I if it satisfies 
$$\limsup_{t \to T} (T-t)\|u(t)\|_{L^\infty} < +\infty,$$
otherwise, the blowup is called Type II. The existence of Type II blowup solutions to \eqref{eq:NLH}  was established in some special range of $p = p(d)$, see for example \cite{FHVslps00}, \cite{Sjfa12}, \cite{PMWZdcds20}, \cite{Hihp20}, \cite{PMWjfa21}, \cite{PMWams19}, \cite{PMWZZarx20} \cite{HVcras94},\cite{Cmams18}, \cite{CMRjams19} and references therein. \\

\paragraph{Singularities in the Keller-Segel system:} We wonder whether the presence of the drift term in \eqref{sys:KS} would create a different blowup mechanism. Indeed, this was the case $d  =2$ where the system \eqref{sys:KS} is said to be $L^1$-critical in the sense that the rescaling function
\begin{equation}
\forall \lambda > 0, \quad u_\lambda(x,t) = \frac{1}{\lambda^2} u\Big(\frac{x}{\lambda}, \frac{t}{\lambda^2}\Big)
\end{equation}
solves the same system and preserves the same $L^1$-norm, 
$$\| u\|_{L^1(\Rb^2)} = \|u_\lambda\|_{L^1(\Rb^2)}.$$
The blowup is completely different from the nonlinear heat equation \eqref{eq:NLH} where only Type I blowup may occur (see \cite{GMSiumj04}). In the two-dimensional case, there is the critical mass threshold $8\pi$ to distinguish between global and blowup solutions. In particular, positive solutions with mass below the $8\pi$ exist for all time and converge to forward self-similar profiles (see \cite{CPZmjm04, DNRjde98, BDPjde06, biler20068pi, BCCjfa12, wei2018global, HYna24}). Solutions with mass $8\pi$ and finite second moment lead to concentration in infinite time \cite{BCMcpam08}, and constructive examples have been obtained in \cite{DPDMWarma24, GMcpam18}. Solutions with mass above $8\pi$ blow up in finite time, and the blowup rate for the single bubble concentration is quantized in \cite{CGMNcpam21, CGMNapde22} (see also \cite{HVma96}, \cite{Vsiam02}, \cite{Sna07} and \cite{RSma14} for earlier results) with
\begin{equation}
\|u(t)\|_{L^\infty(\Rb^2)} \sim \left[ \begin{array}{ll}  c_0(T-t)^{-1} \exp (\sqrt{2 |\ln(T-t)|}), \\
c_\ell(u_0) (T-t)^{-\ell} |\ln (T-t)|^\frac{\ell + 1}{ \ell - 1}, \quad \ell \in \mathbb{N}, \ell\geq 2.
\end{array}      \right. 
\end{equation}
We also have finite-time blowup solutions with a $16 \pi$ mass concentrating at a single point formed by a collision of two sub-collapses in the recent work \cite{CGMNarx24}. In higher dimensional cases $d \geq 3$, we have the existence of blowup solutions  established in \cite{CGMNjfa23} where the authors rigorously constructed a so-called collapsing-ring blowup solutions in the spirit of blowup for the nonlinear Schr\"odinger equation \cite{MRSduke14} and found the blowup rate 
\begin{equation}
\|u(t)\|_{L^\infty(\Rb^d)} \sim (T-t)^{-\frac{2d}{d-1}}. 
\end{equation}
This result was formally derived in \cite{HMVnonl97} for $d = 3$ and reestablished by Brenner \textit{et al} in \cite{BCKSVnon99} for $d \geq 3$ (with a formal analysis too). The authors of \cite{BCKSVnon99} also predicted many other blowup patterns for \eqref{sys:KS} in the higher dimensional cases. In particular, there are countably many backward self-similar solutions to \eqref{eq:urt} (see \cite{HMVjcam98}, \cite{Sfe05}, \cite{GMSarma11} and  \cite{SWCMP19}) of the form 
$$u(r,t) = \frac{1}{T-t}\Phi(y), \quad y = \frac{|x|}{\sqrt{T-t}},$$
where $\Phi$ solves
\begin{equation*}
0 = \Delta_d \Phi + \Big( \frac{1}{y^{d-1}}\int_0^y \Phi(\zeta) \zeta^{d-1} d\zeta \Big) \pa_y \Phi + \Phi^2 - \frac{1}{2}y\pa_y \Phi - \Phi.
\end{equation*}
The explicit solutions are given by
\begin{equation*}
\Phi_1 \equiv 0, \quad \Phi_2 \equiv 1, \quad \Phi_3 = \frac{2(d-2)}{y^2}, \quad \Phi_4 = \frac{1}{y^{d-1}}\pa_y\Big[ \frac{4(d-2)(2d + y^2) y^d}{(2(d-2) + y^2)^2} \Big].
\end{equation*}
In \cite{glogic2024stable}, Glogic and Sch\"orkhuber showed the stability of the solution with $\Phi_4$ profile in $H^3(\Rb^3)$. This stability result is then extended by Collot and Zhang in \cite{CZarxiv24} through a spectral analysis, where the authors established that all these self-similar profiles are stable along a set of initial data with finite Lipschitz codimension equal to the number of unstable eigenmodes of the associated linearized operator.

\subsection{Main result}
In this paper we exhibit a new type of blowup solutions to \eqref{sys:KS} that have not been observed in the literature to our knowledge (the possible occurrence of this solution was briefly mentioned in Section 4.3 of \cite{BCKSVnon99}).  Consider the space dimension 
$$d = 3,\; 4,$$
and restrict to the case of radially symmetric solutions. Then, for any smooth radial function $u \in L^\infty(\Rb^d)$, the potential term is defined as  
\begin{equation}
\pa_r\Kc_u(r) = -\frac{1}{r^{d-1}}\int_0^r u(\zeta) \zeta^{d-1} d\zeta, \quad r = |x|,
\end{equation}
and the system \eqref{sys:KS} is written as a nonlocal semilinear heat equation, 
\begin{equation}\label{eq:urt}
\pa_t u = \Delta_d u +\Big( \frac{1}{r^{d-1}}\int_0^r u(\zeta) \zeta^{d-1}d\zeta\Big) \pa_r u + u^2,
\end{equation}
where $u(t): \Rb_+ \to \Rb$ and $\Delta_d$ is the Laplacian acting on radial functions in $\Rb^d$, i.e.
$$\Delta_d = \pa_r^2 + \frac{d-1}{r}\pa_r.$$
Our main result is the following.  
\begin{theorem}[Existence of finite-time blowup solutions to \eqref{sys:KS}] \label{theo:1} Consider $d = 3, 4$ and let $\ell = \frac{d}{d-2}$ and $\alpha = \frac{d-2}{2d}$.  There exists radially symmetric initial data $u_0 \in L^{\infty}(\Rb_+)$ such that the corresponding solution to System \eqref{sys:KS} blows up in finite-time $T < \infty$ only at the origin and admits the following asymptotic dynamics. 
\begin{itemize}
\item[(i)] \textup{(Inner expansion)} 
\begin{equation}\label{exp:innerIntro}
u\big(y \sqrt{T-t}, t\big) = \frac{1}{T-t}\Big[ 1 - \frac{1}{B_\ell} \frac{\phi_{2\ell} (y) }{|\log (T-t)|} + o \big( \frac{1}{|\log (T-t)|}\big) \Big] \quad \textup{as} \quad t \to T,
\end{equation}
where  the convergence holds on any compact sets $\{ y \leq C\}$, the function $\phi_{2\ell}(y)$ is the polynomial of degree $2\ell$ satisfying $\Delta \phi_{2\ell} - \frac{1}{2\ell} y \pa_y \phi_{2\ell} + \phi_{2\ell} = 0$,
\begin{equation}\label{eq:B3B2Intro}
B_3 = 39360 \;\; \textup{for}\;\; (d, \ell) = (3,3) \quad \textup{and} \quad  B_2 = 576 \;\;\textup{for}\;\; (d, \ell) = (4,2).
\end{equation}

\item[(ii)] \textup{(Intermediate profile)} Let $Q$ be a positive function defined by 
\begin{equation}\label{def:QIntro}
\forall \xi \in \Rb_+, \quad  \frac{1 - dQ}{Q^\ell} = c_\ell \xi^{2\ell} \quad \textup{with} \quad  c_\ell = \frac{d^{\ell + 1}}{B_\ell(d + 2\ell)\ell^\ell} > 0.
\end{equation} 
Let $F(\xi) = dQ(\xi) + \xi Q'(\xi)$, we have 
\begin{equation} \label{exp:interProfileIntro}
\sup_{|y| \in \Rb^d} \left|(T-t) u\big(y\sqrt{T-t}, t\big) -  F\Big( \frac{y}{|\log (T-t)|^\frac{1}{2\ell}} \Big)  \right| \to 0 \quad \textup{as} \;\; t \to T.  
\end{equation}

\item[(iii)] \textup{(Final profile)} There exists $u^* \in \Cc(\Rb_+ \setminus \{0\}, \Rb)$ such that $u(r,t) \to u^*(r)$ as $t \to T$ uniformly on compact subsets of $\Rb_+\setminus \{0\}$, where 
\begin{equation}\label{exp:finalprofileIntro}
u^*(r) \sim  (d- 2) \left(\frac{2}{c_\ell} \right)^\frac{1}{\ell} \frac{|\log r| ^ \frac{1}{\ell}}{r^2} \quad \textup{as}\;\; r \to 0.
\end{equation}
\end{itemize}
\end{theorem}

\begin{remark}[New blowup profile] One of the significant contribution of this work is the construction of a \textit{new} blowup profile \eqref{def:QIntro} with a log correction to the blowup variable 
\begin{equation} \label{def:xiIntro}
\xi  =  \frac{r}{\sqrt{T-t} \; |\log(T-t)|^\frac{1}{2\ell}} \quad \textup{with} \quad \ell \geq 2.
\end{equation}
Let us mention that Brenner \textit{et al} \cite{BCKSVnon99} wondered whether blowup solutions exist with this particular scaling for any dimension $d \geq 3$, not necessarily with the same power we get here for the log correction  (see Section 4.3 in that paper). In particular, they did not provide the exact power of the log correction (i.e. $\frac{1}{2\ell}$), as we found here for $d = 3$ or $d = 4$.  Since the logarithmic correction in \eqref{def:xiIntro} is a consequence of the appearance of the zero eigenvalue in the spectrum of the linearized operator $\Hs_{\frac 1{2\ell}} + \textup{Id}$ in the radial setting (see \eqref{def:H12lintro}), which is given by $\big\{\lambda_k = 1 - \frac{k}{\ell}, k \in \mathbb{N}\big\}$, we believe that the blowup mechanism described in Theorem \ref{theo:1} can not appear in the case $d \geq 5$ as $k = \ell = \frac{d}{d-2} \not \in \mathbb{N}$ if $d \geq 5$ (see Remark \ref{remark:specofLs}). 

The appearance of the \textit{new} blowup scale given in \eqref{def:xiIntro} shows a strong influence of the drift-term to the blowup dynamic of \eqref{sys:KS}. Recall from \eqref{eq:profileNLH} that the stable blowup scale for \eqref{eq:NLH} is given with $\ell = 1$ and the intermediate profile is explicitly given.  As a consequence, the existing analysis developed for \eqref{eq:NLH} can not be straightforward implemented to \eqref{sys:KS}, although the general framework for the construction remains the same once the blowup is concerned.

It is worth remarking that the final blowup profile derived in \cite{SWCMP19} for the class of radially symmetric \textit{decreasing} solutions satisfies
$$
c_1r^{-2} \leq u(r,T) \leq c_2 r^{-2} \quad \textup{and} \quad u(x,T)  \leq C(T-t + r^2)^{-1}.
$$
Recall that our constructed solution is radially symmetric, but not a decreasing function. Indeed, the inner expansion \eqref{exp:interProfileIntro} involves a Hermite type polynomial of degree $2\ell$ which changes signs on the set $y \in (0, y_0)$ for $y_0\gg 1$. Hence, our intermediate and final blowup profiles \eqref{exp:interProfileIntro} and \eqref{exp:finalprofileIntro} are excluded from what described in \cite{SWCMP19} and in agreement with the description of \cite{GMSarma11} asserting that all Type I blowup solutions are asymptotically backward self-similar. 
\end{remark}

\begin{remark}[Co-dimensional stability] The initial data we consider in the construction depends on $\ell$ parameters $(d_i)_{0 \leq i \leq \ell-1}$ (see \eqref{def:intitialdata_q} for a proper definition) to control $\ell$ growing eigenmodes of the linearized operator $\Hs_{\frac{1}{2\ell}} + \textup{Id}$ in the radial setting (see \eqref{def:H12lintro} for its definition). One parameter can be eliminated by the translation in time invariance of the problem, so it remains $(\ell - 1)$ eigenmodes to be handled. Roughly speaking, our constructed solution is $(\ell - 1)$ co-dimension stable in the sense that if we fix those $(\ell-1)$ unstable directions and perturb only the remaining components of the solution (see Definition \eqref{Definition-shrinking -set} for a definition of solution decomposition and a bootstrap regime to control them), the solution still admits the same behavior as described in Theorem \ref{theo:1}.
We expect a similar co-dimensional stability still holds in the non-radial setting, where the spectrum of the linearized operator $\Hs_{\frac{1}{2\ell}} + \textup{Id}$ is given by $\big\{\lambda_k = 1 - \frac{k}{2\ell}, k \in \mathbb{N}\big\}$, where $\lambda_k$ has the multiplicity $m(\lambda_k) = m(k,d)$. Hence, we have a total of $\sum_{k = 0}^{2\ell} m(k,d)$ unstable and null modes, in which we can eliminate $1 + d$ unstable directions thanks to the time and space translation invariance of the problem.
\end{remark}

\begin{remark}[Extensions and related problems] From our analysis, we suspect that the blowup scale with a log correction \eqref{def:xiIntro} can only occur in dimensions $d = 3$ and $d = 4$. In fact, the appearance of such a log correction is strongly related from our point of view to the presence of a zero eigenvalue for the linearized operator defined below in \eqref{def:Ls_alp}, which occurs only for $d = 3,4$ (see Remark \ref{remark:specofLs}). Other blowup scales without a log correction are suspected to exist similarly as in the nonlinear heat equation \eqref{eq:NLH} where a negative eigenmode of the linearized operator is assumed to be dominant in the inner expansion \eqref{exp:innerIntro}. 

The suppression of blowup in \eqref{sys:KS} by modifying the nonlinearity has been an interesting direction recently, see for an example \cite{KXarma16}, \cite{BHsiam17}, \cite{HTarma19}, \cite{IXZtams21}. We remark that our construction actually works for the nonlinear perturbation problem where we consider the nonlinearity $u^2$ added a perturbation $f(u)$, namely the problem 
$$\pa_t u= \Delta u - \nabla \Kc_u \cdot \nabla u + u^2 + f(u),$$
where $f$ satisfies the growth condition 
$$|f(u)| \leq C (1 + |u|^q)\;\; \textup{for}\;\; 0 \leq q < 2, \quad \textup{or} \quad f(u) = \epsilon u^2 \;\; \textup{for}\;\; 1 + \epsilon > 0.$$
The first assumption turns to be exponentially small in the self-similarity setting  \eqref{def:selfsimilarity} and the contribution from this small nonlinear term is neglectable. The later assumption ensures that the associated ODE $u' = (1 + \epsilon)u^2$ still blows up in finite time, hence the intermediate blowup profile \eqref{exp:finalprofileIntro} and the final profile \eqref{exp:finalprofileIntro} are modified with the factor $\frac{1}{1 + \epsilon}$. 
We suspect there is a critical value $\epsilon_* = \epsilon_*(d) > 0$ so that for $\epsilon \leq - \epsilon_*$ would prevent blowup to happen, or if blowup does occur, its dynamic would be completely different from what established in this paper. 

The analysis presented in this work is expected to be applicable to the general equation \eqref{eq:generalequation} with $A(u) = u$ and $B(u) = u^{p-1}$ for $p > 1$ or $A(u) = u^m$ for $m > 0$ and $B(u) = u$ up to some technicalities. The later case is an interesting model (\textit{porous medium} type equation) used to describe gravitational collapse phenomena (see \cite{CSPRE11}).

It is worth mentioning that the problem \eqref{sys:KS} in the radial setting has a similar form to a nonlocal parabolic problem formally analyzed in \cite{Gjmp08} (see (3.1) and (3.2)) arising in the study of blowup singularities in the 3D axisymmetric Navier-Stokes equation. We expect the framework developed in this present work would be implemented for a rigorous analysis initiated in \cite{Gjmp08}. 
\end{remark}

\subsection{Strategy of the analysis}
\noindent We briefly describe the idea of the proof of Theorem \ref{theo:1} consisting of the following steps. \\

\noindent - \textit{Renormalization in the self-similar variables:} We consider the change of variables 
\begin{equation*}
u(r,t) = \frac{1}{T-t}w(y,s), \;\; y = \frac{r}{\sqrt{T-t}}, \;\; s = -\log(T-t),
\end{equation*}
and introduce the partial mass setting 
$$v(y,s) = \frac{1}{y^d}\int_0^y w(\zeta, s)\zeta^{d-1}d\zeta, \qquad w = \frac{1}{y^{d-1}}\pa_y(y^d v).$$
where $v$ solves the semilinear heat equation 
\begin{equation}
\pa_s v = \Delta_{d+2} v -\frac{1}{2}y\pa_y v + dv^2 + yv \pa_y v. 
\end{equation}
We note that such a transformation is just to simplify the analysis, and emphasize that the strategy and main idea remain the same once we work with the equation satisfied by $w$. \\

\noindent - \textit{Linearization}: Through a formal computation of the blowup profile $Q$ given in Section \ref{sec:blowupprofile}, we introduce the linearization 
$$v(y,s) = Q(\xi) + \vep(y,s), \quad \xi = \frac{y}{s^\frac{1}{2\ell}},$$
where $\vep$ solves the linearized problem 
\begin{equation}
\pa_s \vep = \Hs \vep + NL(\vep) + E,
\end{equation}
where $NL$ is a quadratic nonlinear term, $E$ is a generated error and $\Hs$ is the linearized operator
$$\Hs = \Delta_{d+2} - \Big(\frac{1}{2} - Q(\xi)\Big)y\pa_y  + \big(2dQ -1 + \xi \pa_\xi Q(\xi)\big).$$
We observe that $\Hs$ behaves differently depending on the behavior of the profile $Q(\xi)$: \\
- For $y \gg s^\frac{1}{2\ell} (\xi \gg 1)$ , we have by the decaying property $|Q(\xi)| + |\xi \pa_\xi Q(\xi)| = \Oc(\xi^{-2})$,  the linear operator $\Hs$ behaves like $\Delta_{d+2} - \frac{1}{2}y\pa_y  - \textup{Id}$, which has fully negative spectrum. \\
- For $y \ll s^\frac{1}{2\ell} (\xi \ll 1)$, we have by the asymptotic behavior $\big|Q(\xi) - \frac{1}{d}\big| + |\xi \pa_\xi Q(\xi)|  = \Oc(\xi^2)$,  the linear operator $\Hs$ behaves like $\Delta_{d+2} - \frac{1}{2\ell }y\pa_y  + \textup{Id}$, which has $\ell$ positive eigenvalues, a zero eigenvalue and infinity many negative ones. \\
- For $y \sim s^\frac{1}{2\ell} (\xi \sim 1)$, this is the transition region (intermediate zone) that we do not have any asymptotic simplification of $\Hs$. This is one of the major difficulties of the paper.  Indeed, it makes one of the main differences with respect to the analysis for the nonlinear heat equation \eqref{eq:NLH} that results in a different approach presented in this paper. \\

\noindent - \textit{Decomposition and control of the flow in three different regions:} Based on the behavior of the linearized operator $\Hs$, we split the control of $\vep$ into three regions: for a fixed large constant $K \gg 1$,  \\
- The outer region $y \geq K s^\frac{1}{2\ell} (\xi \geq K)$: Since $\Hs$ behaves like the one with fully negative spectrum, the estimate of $\vep$ in this region is straightforward by using the semigroup associated to the linear operator $\Delta_{d+2} - \frac{1}{2}y \pa_y$ (see Section \ref{sec:outer}), 
\begin{equation*}
j =0, 1, \quad \| (y \pa_y)^j \vep(y,s) \mathbf{1}_{\{\xi \geq K\}}\|_{L^\infty} \lesssim \|(y \pa_y)^jE(s)\|_{L^\infty} + \| (y \pa_y)^j\vep(y,s) \mathbf{1}_{\{\xi \sim K\}}\|_{L^\infty},
\end{equation*}  
where $\|(y \pa_y)^jE(s)\|_{L^\infty} \lesssim s^{-\frac{1}{\ell}}$ is the typical size of the generated error. We need the information from the intermediate region for the boundary term located on $\xi \sim K$ to completely close the estimate in the outer region. \\
- The intermediate region $K \leq y \leq 2K s^\frac{1}{2\ell}$, we control the solution in the weighted $L^2$ norm, 
$$\|\vep(s)\|_\flat^2 = \int_K^\infty  \frac{|\vep(y,s)|^2}{y^{4\ell + 2}} \frac{dy}{y}.$$
(we can replace the weight $y^{4\ell + 2}$ by $y^{2k}$ for any $k \geq 2\ell +1$ with an improved refinement of the generated error). Thanks to the monotone property of the profile $Q$ and the dissipative structure of the parabolic equation, we are able to arrive at the monotonicity formula  (see Lemma \ref{lemm:mid})
\begin{equation}
j = 0, 1, 2, \quad  \frac{d}{ds} \|(y\pa_y)^j \vep\|_\flat^2 \leq -\delta_0 \|(y\pa_y)^j \vep\|_\flat^2 + \|(y\pa_y)^j E\|_\flat^2  + \|(y\pa_y)^j \vep \mathbf{1}_{ y \sim K}\|_\flat^2, 
\end{equation}
where $\|(y\pa_y)^j E\|_\flat^2 \lesssim s^{-2 - \frac{3}{\ell}}$ is the size of the error term. By Sobolev inequality, we obtain a pointwise estimate $|\vep(y,s)| \lesssim s^{-1 - \frac{3}{2\ell}} (|y|^{2\ell + 1} + 1)$ that provides the necessary  information of $\vep$ at the boundary $y \sim s^\frac{1}{2\ell}$ to complete the estimate in the outer region. It is  worth mentioning that in the case of the nonlinear heat equation \eqref{eq:NLH}, this kind of pointwise estimate can be directly achieved by using a semigroup approach. However, we are not able to follow that approach for the Keller-Segel equation \eqref{sys:KS} due to the lack of knowledge on semigroup theory associated to $\Hs$. Here, we still need the information of $\vep$ at the boundary $y \sim K$ to completely close the estimate of $ \|(y\pa_y)^j \vep\|_\flat$ after a forward integration in time. \\
- The inner region $y \leq 2K$, the linearized operator $\Hs$ is regarded as a perturbation of
\begin{equation} \label{def:H12lintro}
\Hs_\frac{1}{2\ell} + \textup{Id} \quad \textup{with} \quad   \Hs_\frac{1}{2\ell}:=\Delta_{d+2} - \frac{1}{2\ell} y\pa_y.
\end{equation}
The operator $\Hs_\frac{1}{2\ell}$ is self-adjoint in $L^2_\rho$ with the exponential weight $\rho = \exp( - \frac{y^2}{4\ell}) y^{d+1}$. In the radial setting, we recall that $\Hs_\frac{1}{2\ell} + \textup{Id}$ possesses $\ell$ positive eigenvalues, a zero eigenvalue, and an infinity many negative ones, we further decompose 
$$\vep(y,s) = \vep_\natural(y,s) + \tilde \vep(y,s), \quad \vep_\natural(y,s) = \sum_{k = 0}^{2\ell-1} \vep_k(s) \varphi_{2k}(y), \quad \langle \tilde{\vep},\varphi_{2k}\rangle_{L^2_\rho} =0 \; \textup{for} \; k = 0, \cdots, 2\ell-1, $$
where $\varphi_{2k}$ is the eigenfunction of $\Hs_\frac{1}{2\ell}$ corresponding to the eigenvalue $- \frac{k}{\ell}$ and $\tilde{\vep}$ solves the equation 
$$\pa_s \tilde{\vep} = \Hs_\frac{1}{2\ell}\tilde{\vep} + \tilde{\vep}  + \sum_{k = 0}^{2\ell -1} \big[-\vep_k' + \big( 1 - \frac{k}{\ell}\big) \vep_k \big] \varphi_{2k}(y) + R + \tilde{\Vc}(\tilde \vep) + NL(\tilde \vep),$$
where $\tilde \Vc$ and $NL$ are small linear and nonlinear terms. Using the spectral gap $\langle \Hs_\frac{1}{2\ell} \tilde \vep, \tilde{\vep} \rangle_{L^2_\rho} \leq -2 \| \tilde{ \vep}\|^2_{L^2_\rho}$ and a standard energy estimate, we end up with (see Lemma \ref{lemm:L2rho})
$$ \frac{d}{ds}\| \tilde{ \vep}\|^2_{L^2_\rho} \leq - \delta \| \tilde{ \vep}\|^2_{L^2_\rho} + \| R \|^2_{L^2_\rho}, \quad \textup{with} \quad \| R\|_{L^2_\rho} \lesssim s^{-3}.$$
As for the finite-dimensional part, we simply obtain by a projection onto the eigenmode $\varphi_{2k}$, 
$$\big|-\vep_k' + \big( 1 - \frac{k}{\ell}\big) \vep_k \big| \lesssim \| \tilde{\vep}\|_{L^2_\rho} + s^{-2}, \quad  \big|\vep_\ell' + \frac{2}{s}\vep_\ell \big| \lesssim \| \tilde{\vep}\|_{L^2_\rho} + s^{-3}. $$ 
The equation of $\vep_\ell$ is delicate as it is related to the projection onto the null mode where the contribution from the small potential term must be taken into account to produce the factor $\frac{2}{s}$ as well as an algebraic cancellation in the projection of the error term onto the null mode to reach $\Oc(s^{-3})$. Those calculations get used of the precise value $B_\ell$ given in \eqref{eq:B3B2Intro} (see Lemma \ref{lemm:finitepart}).  A forward integration in time yields the estimate 
$$\|\tilde \vep\|_{L^2_\rho} \lesssim s^{-3}, \quad |\vep_k(s)| \lesssim s^{-2} \; \textup{for} \;\ell + 1 \leq k \leq 2\ell -1, \quad |\vep_\ell(s)| \lesssim \frac{\log s}{s^2}.$$
The remaining growing modes $\big(\vep_k(s)\big)_{0 \leq k \leq \ell-1}$ are then controlled by a topological argument where we need to construct initial data for which these components converge to zeros as $s \to \infty$. The $L^2_\rho$ estimate provides information on compact sets of $y$, where we can get the estimate of $\vep(y,s)$ for $y \sim K$ to close the estimate for the intermediate region. This decomposition is detailed in the definition of bootstrap regime \eqref{Definition-shrinking -set} in which we successfully construct solutions admitting the behavior as described in Theorem \ref{theo:1}. \\

We organize the rest of the paper as follows: In Section \ref{sec:blowupprofile} we perform a formal spectral analysis to obtain the blowup profile. We formulate the linearized problem in Section \ref{sec:linear} and define a bootstrap regime to control the remainder. In Section \ref{sec:ControlBs} we control the remainder in the bootstrap regime and prove that the solution of the linearized problem is trapped in this regime for all time from which we conclude the proof of the main theorem. \\

\medskip

\paragraph{Acknowlegments:} V.T. Nguyen is supported by the National Science and Technology Council of Taiwan (ref. 111-2115-M-002-012 and ref. 112-2628-M-002-006). A part of this work was done when V.T. Nguyen visited the Universit\'e Paris Dauphine and he wants to thank the institution for their hospitality and support during his visit. H. Zaag is supported by the ERC Advanced Grant LFAG/266
“Singularities for waves and fluids”. The authors would also like to thank the anonymous referees for their helpful comments and suggestions.

\section{Formal derivation of the blowup profile} \label{sec:blowupprofile}
In this section we formally derive the blowup profile through a spectral approach that will be rigorously implemented in the next section. As you will see, the profile with a logarithmic correction to the blowup scale only appears in the case of dimension $d = 3$ or $d = 4$ for which the linearized operator poses zero eigenvalues. We work with the self-similar variables
\begin{equation}\label{def:selfsimilarity}
u(r,t) = \frac{1}{(T-t)} w(y,s), \quad \Kc_u(x,t) = \Kc_w(y,s), \quad y = \frac{r}{\sqrt{T-t}}, \quad s = - \log(T-t),
\end{equation}
where $w$ solves the equation for $ (y,s) \in \Rb_+ \times [-\log T, +\infty)$,
\begin{equation}\label{eq:wys}
\pa_s w = \Delta_d w - \left( \pa_y\Kc_w + \frac{1}{2} y \right)\pa_y w  - w + w^2, 
\end{equation}
with $\Delta_d$ being the Laplacian in $\Rb^d$ acting on radial functions, i.e.  
$$\Delta_d = \pa_y^2 + \frac{d-1}{y}\pa_y,$$
and the potential term is defined by 
\begin{equation}\label{eq:exprePhiy}
\pa_y \Kc_w(y,s) = - \frac{1}{y^{d-1}}\int_0^y w(\zeta,s) \zeta^{d-1} d\zeta. 
\end{equation}
A linearization of $w$ around the nonzero constant solution 
$$ \bar w = 1, \qquad -\pa_y \Kc_{\bar w}= \frac{y}{d},$$
leads to the linearized equation  for $q = w - \bar w, \quad \Kc_q = \Kc_w - \Kc_{\bar w}$:
 \begin{equation}\label{eq:qys}
\pa_s q = \Ls_\alpha q  - \pa_y \Kc_q \pa_y q + q^2, \quad (y,s) \in \Rb_+ \times [-\log T, +\infty), 
\end{equation}
where 
\begin{equation}\label{def:Ls_alp}
\Ls_\alpha = \Delta_d - \alpha\; y \pa_y +1 \quad \textup{with} \quad  \alpha = \frac{d-2}{2d}. 
\end{equation}
\paragraph{Spectral properties of $\Ls_\alpha$:} The linear operator $\Ls_\alpha$ is formally a self-adjoint operator in $L^2_\omega(\Rb_+)$ with the weight function 
\begin{equation}\label{def:omega}
\omega(y) = c_d y^{d-1} e^{-\frac{\alpha |y|^2}{2}},\qquad \int_0^\infty \omega (y) dy=1,
\end{equation}
where $c_d = 2^{\frac d2-1}\alpha^{-\frac d2}\Gamma \left (\frac d2\right )$ to normalize $\|\omega\|_{L^1} = 1$, $\Gamma$ is the Gamma function. The  spectrum of $\Ls_\alpha$ is discrete 
\begin{equation}\label{def:specLsalpha}
\textup{spec}(\Ls_\alpha) = \{ \lambda_{2n} = 1 - 2 n \alpha, \; n \in \mathbb{N}\},
\end{equation}
and the corresponding eigenfunction $\phi_{2n}$ is a polynomial of degree $2n$ that satisfies
\begin{equation}\label{eq:phi2n}
\Delta_d \phi_{2n} - \alpha y \pa_y \phi_{2n}= -2n \alpha \phi_{2n}.
\end{equation}
In particular, we have the close form of $\phi_{2n}$ given by 
\begin{equation}\label{def:phi2ell}
\phi_{2n}(y)  = H_n(z) =  \sum_{k = 0}^n A_{n, k} z^{k}  \quad \textup{with} \quad  z =2\alpha y^2,
\end{equation}
where $H_n$ is the regular solution to the Kummer-type ODE 
\begin{equation}\label{eq:Hz}
4zH_n'' + \big(2d- z\big) H_n' + n H_n = 0. 
\end{equation}
and $A_{n,k}$'s satisfy the recurrence relation
\begin{equation}\label{def:Ank}
 A_{n,n} = 1, \quad  A_{n, k - 1} =  - \frac{2k (2k + d - 2)}{n - k + 1} \; A_{n, k} \;\; \textup{for}\;\; k = 1, ..., n.
 \end{equation}
We note the orthogonality identity 
\begin{equation}\label{eq:ortho}
\int_0^\infty \phi_{2n}(y) \phi_{2m}(y) \omega(y) dy = (2\alpha)^{-\frac{d-2}{2}}\int_0^\infty H_n(z)H_m(z) z^{\frac{d-1}{2}} e^{-\frac{z}{4}}  dz =    a_{n}\delta_{n,m},  
\end{equation}
where $\delta_{n,m} = 0$ if $n \ne m$ and $\delta_{n,m} = 1$ if $n = m$. 

\remark[The zero mode] \label{remark:specofLs} Consider $\ell \in \mathbb{N}$ such that $\lambda_{2\ell} = 1 - 2 \alpha \ell = 0$, which gives $\ell = \frac{1}{2\alpha} = \frac{d}{d  -2} \in (1,3]$. We see that there are only two cases for which $\ell$ is an integer number:
\begin{equation*}
\ell = 3 \;\; \textup{for}\;\; d = 3 \quad \textup{and}  \quad \ell = 2 \;\; \textup{for}\;\; d = 4.
\end{equation*}
\remark[The first few eigenfucntions] We list here the first few eigenfunctions (generated by Matlab symbolic) served for our computation later:\\
-  for $d = 3$,
\begin{align*}
&H_0(z) = 1, \qquad \quad H_1(z) = z - 6, \qquad \quad H_2(z) = z^2 - 20 z + 60,\\
& H_3(z) = z^3 - 42z^2 + 420 z - 840,\\
&H_4(z) = z^4 - 72z^3 + 1512z^2 - 10080z + 15120,\\
& H_5(z) = z^5 - 110z^4 + 3960z^3 - 55440z^2 + 277200z - 332640,\\
& H_6(z) = z^6 - 156z^5 + 8580z^4 - 205920 z^3 + 2162160 z^2 - 8648640z + 8648640,
\end{align*}
- for $d = 4$,
\begin{align*}
&H_0(z) = 1, \quad H_1(z) = z - 8, \quad H_2(z) = z^2 - 24 z + 96, \\
&H_3(z) = z^3 - 48z^2 + 576z - 1536,\\
&H_4(z) = z^4 - 80z^3 + 1920z^2 - 15360z + 30720.
\end{align*}
We can expand an arbitrary polynomial  $P_n(z)$ in terms of $\sum_{k = 0}^nH_k(z)$ through the inverse
\begin{equation}
\begin{pmatrix}
1\\ z \\ z^2\\ \vdots \\ z^n
\end{pmatrix} = \Dc_{n}^{-1} \begin{pmatrix}
H_0\\ H_1 \\ H_2 \\ \vdots \\ H_n
\end{pmatrix}   \quad \textup{with} \quad \Dc_{n} = \begin{pmatrix}
1\\
A_{1,0} & 1\\
A_{2,0}& A_{2,1} & 1\\
\vdots & \vdots &  & \ddots\\
A_{n,0} & A_{n, 1} & A_{n, 2}& \cdots & 1
\end{pmatrix},
\end{equation}
where $A_{i,j}$ is given by \eqref{def:Ank}.  A direct  check yields
\begin{equation}\label{re:zntoHn}
\Dc_n^{-1} =\{ |A_{i,j}|\}_{1 \leq i,j \leq n}, \quad z^n = \sum_{k = 0}^n |A_{n,k}| H_k(z),
\end{equation}
from which and the orthogonality \eqref{eq:ortho} imply
\begin{equation}
\int_0^\infty y^{2n} \phi_{2m}(y) \omega(y) dy = 0 \quad \textup{for}\;\; m \geq n+1. 
\end{equation}

\bigskip

\paragraph{Appearance of Type I log-profile:} Consider  
$$(d, \ell) = (3,3) \quad \textup{and} \quad (d, \ell) = (4,2).$$
We  decompose $q(y,s)$ according to the eigenspace of $\Ls_\alpha$, namely that
\begin{align*}
q(y,s) & = \sum_{k \in \mathbb{N} } a_{k}(s)\phi_{2k}(y) \equiv \sum_{k \in \mathbb{N}} a_k(s) H_k(z), \quad z = 2\alpha y^2.
\end{align*} 
Since $\sum_{k \geq \ell + 1} a_{k}(s) \phi_{2k}(y)$ is the projection of $q(y,s)$ on the negative mode of $\Ls_\alpha$, we may  ignore it in this formal derivation, meaning that we only consider the ansatz
\begin{equation}\label{decomp:qtruncated}
q(y,s) = \sum_{k \leq \ell}a_{k}(s) \phi_{2k}(y).
\end{equation}
By assuming  the zero mode is dominant in the sense that 
\begin{equation}\label{eq:assumall}
|a_{k} (s)| \ll |a_{\ell}(s)| \;\; \textup{for}\;\; k \leq \ell -1,
\end{equation}
we  plugin this ansatz  into \eqref{eq:qys} and project it onto the eigenmode $\phi_{2j}$ for $j = 0,..., \ell$:
\begin{align*}
a_{j}' = (1 - 2j \alpha) a_{j} + \| \phi_{2j} \|_{L^2_\omega}^{-2} \left \langle NL, \phi_{2j}\right\rangle_{L^2_\omega},
\end{align*}
where 
$$NL = -\Big(\sum_{k \leq \ell} a_{k} \pa_y \phi_{2k}\Big)\Big(\sum_{k \leq \ell}a_{k}\pa_y\Kc_{\phi_{2k}} \Big) + \Big(\sum_{k \leq \ell}a_{k}\phi_{2k}\Big)^2.$$
From \eqref{eq:assumall} and the fact the $\int_0^\infty P_n(y) \omega(y)dy = \Oc(1)$, we see that 
$$\| \phi_{2j} \|_{L^2_\omega}^{-2}\left \langle NL, \phi_{2j}\right\rangle_{L^2_\omega} = \Oc\big(a_{\ell}^2\big) \quad \textup{and} \quad \| \phi_{2\ell} \|_{L^2_\omega}^{-2}\left \langle NL, \phi_{2\ell}\right\rangle_{L^2_\omega} = a_{\ell}^2 B_{\ell} + o\big(a_{\ell}^2\big),$$
where 
$$B_{\ell} =  \| \phi_{2\ell} \|_{L^2_\omega}^{-2}\left \langle - \pa_y\phi_{2\ell} \pa_y \Kc_{\phi_{2\ell}} + \phi_{2\ell}^2, \phi_{2\ell}\right\rangle_{L^2_\omega}.$$
To compute $B_{\ell}$, we simply expand 
\begin{equation}\label{eq:expNL}
- \pa_y\phi_{2\ell} \pa_y \Kc_{\phi_{2\ell}} + \phi_{2\ell}^2 = \sum_{k = 0}^{2\ell} B_{k} \phi_{2k}(r) \equiv \sum_{k = 0}^{2\ell} B_{k} H_k(z),
 \end{equation}
from which we directly obtain the constant $B_\ell$ by the orthogonality  \eqref{eq:ortho}. \\

\label{sec:computeBell}
\paragraph{Compute the constant $B_\ell$ in \eqref{eq:expNL}:} From the definition $\phi_{2\ell}(y) = H_\ell(z)$ with $z = 2\alpha r^2$, we have
\begin{equation*}
8\alpha z^{\frac{d}{2}} \pa_z \Kc_{H_\ell} = - \int_0^z H_\ell(\xi) \xi^{\frac{d-2}{2}} d\xi. 
\end{equation*} 
We then write
\begin{align*}
- \pa_y\phi_{2\ell} \pa_y \Kc_{\phi_{2\ell}} + \phi_{2\ell}^2 &= \frac{1}{y^{d-1}} \pa_y \big(  y^{d-1}  \phi_{2\ell}\pa_y \Kc_{\phi_{2\ell}} \big) =  \frac{1}{ z^{\frac{d-2}{2}}}\pa_z \left(H_\ell \int_0^z H_\ell(\xi) \xi^{\frac{d-2}{2}} d\xi \right).
\end{align*}
For the case  $(d =3, \ell = 3)$, we have 
\begin{align}
- \pa_y\phi_{2\ell} \pa_y \Kc_{\phi_{2\ell}} + \phi_{2\ell}^2 &= \frac{1}{\sqrt{z}} \pa_z \left( H_3 \int_0^z H_3(\xi) \sqrt{\xi} d\xi\right) \nonumber\\
& = \frac 53 z^6 - \frac{416}{3}z^5 + \frac{12628}{3}z^4 - 57792z^3 + 364560 z^2 - 940800 z + 705600 \nonumber \\
& =  \sum_{k = 0}^6 B_k H_6(z),\label{eq:B3ex}
\end{align}
where the second line is computed by Matlab symbolic, and we have from \eqref{re:zntoHn} the value of $B_3$ is given by 
\begin{equation}\label{def:valueB3}
B_3 = \frac{5}{3}|A_{6,3}| - \frac{416}{3}|A_{5,3}| + \frac{12628}{3}|A_{4,3}| - 57792 = 39360. 
\end{equation}
Thus, the ODE satisfied by $a_3$ is 
\begin{equation}
a_3' \sim B_3 a_3^2,  \qquad \textup{hence,}\quad  a_3(s) \sim -\frac{1}{B_3 s}.
\end{equation}
For the case $(d = 4, \ell = 2)$, we similarly compute 
\begin{align}
- \pa_y\phi_{2\ell} \pa_y \Kc_{\phi_{2\ell}} + \phi_{2\ell}^2 &= \frac{1}{z} \pa_z \left( H_2 \int_0^z H_2(\xi) \xi d\xi\right) \nonumber\\
& = \frac 32 z^4 - 70 z^3 + 1056 z^2 - 5760  z + 9216 \nonumber \\
& = \frac{3}{2}H_4 + B_3 H_3 + B_2 H_2 + \cdots, \label{eq:B2ex}
\end{align}
where we use \eqref{re:zntoHn} to compute the value of $B_2$ which is given by
\begin{equation}\label{def:valueB2}
B_2 = \frac{3}{2} |A_{4,2}| - 70 |A_{3,2}| + 1056= 576.
\end{equation}
Hence, 
\begin{equation}
a_2' \sim 576 a_2^2  \qquad \textup{hence,}\quad   a_2(s) \sim -\frac{1}{576 s}.
\end{equation}
In summary, we have found the asymptotic expansion in $L^2_\omega$ as follows: 
\begin{align}\label{eq:expinn}
 w(y,s) = 1 - \frac{1}{B_\ell s} \phi_{2\ell}(y) + \cdots = 1 - \frac{1}{B_\ell} \frac{(2\alpha y^2)^{\ell}}{s} + \cdots,
\end{align}
where $B_\ell$ is given in \eqref{def:valueB3} and \eqref{def:valueB2} for the case $(d = 3, \ell =3)$ and ($d = 4, \ell  = 2$) respectively. \\

\medskip

\paragraph{Matching asymptotic expansions and the profile:} The inner expansion \eqref{eq:expinn} suggests  to look for a profile of the form 
\begin{equation} \label{def:blowupvar}
w(y,s) \sim F(\xi) \quad \textup{with} \quad \xi = \frac{y}{s^\frac{1}{2\ell}} = \frac{|x|} {\sqrt{T-t} |\log (T-t)|^\frac{1}{2\ell}}. 
\end{equation}
Plugging this form to \eqref{eq:wys}, we obtain at the leading order the nonlocal ODE satisfied by $F$:
\begin{equation} \label{eq:odeF}
\frac{1}{\xi^{d-1}} \pa_\xi \Big( F \int_0^\xi F(\zeta) \zeta^{d-1} d\zeta \Big)  - \frac{1}{2}\xi F' - F = 0,
\end{equation}
subjected to the initial condition 
$$\quad F(0) = 1.$$
Since the original problem \eqref{sys:KS} has the divergence structure,
we look for a solution $F(\xi)$ through the partial mass transformation:
$$ Q(\xi) = \frac{1}{\xi^d}\int_0^\xi F(\zeta)\zeta^{d-1} d\zeta, \quad \textup{which is equivalent to} \;\; F(\xi) = \frac{1}{\xi^{d-1}} \pa_\xi (\xi^d Q).$$
By noticing that $ \xi \partial_\xi F = \frac{1}{\xi^{d-1}} \partial_\xi \big(\xi^d F\big) - dF,$ we rewrite the ODE \eqref{eq:odeF} as
$$0 = \partial_\xi\Big(\xi^d F Q - \frac{1}{2}\xi^d F + \big( \frac d2 - 1\big) \xi^d Q \Big).$$
Since we are looking for a regular profile $F$, so $Q(\xi) = \frac{1}{\xi^{d}}\int_0^\xi F(\zeta) \zeta^{d-1} d\zeta \lesssim 1$, an integration yields the ODE
\begin{align}\label{equation-Q}
dQ\big(Q - 1/d \big) + \big(Q - 1/2 \big) \xi Q' = 0, \quad Q(0) = \frac{1}{d}.
\end{align}
The two trivial solutions are $Q = 0$ and $Q = \frac{1}{d} < 1/2$ for $d \geq 3$. 

We observe from \eqref{equation-Q} that $Q(\xi)$ is a positive decreasing function on $(0,\infty)$, namely that 
\begin{align} \label{prop:Q}
    Q(\xi)>0, \; Q'(\xi)<0\mbox{   for   }\xi \in (0,\infty).
\end{align}
As for the positivity, we assume there exists a first point $\xi^*\in (0,\infty)$ such that $Q(\xi^*)=0$ and $Q(\xi)>0$ for $\xi \in (0,\xi^*)$. From \eqref{equation-Q}, we have $Q'(\xi^*)=0$. On the other hand, we note from \eqref{equation-Q} that $Q(\xi) < \frac{1}{2}$ for all $\xi \geq 0$, we then rewrite \eqref{equation-Q} as $Q'(\xi) = f(Q,\xi)$, where $f(Q,\xi) = \frac{dQ(Q - 1/d)}{\xi (1/2 - Q)}$ is a continuous differentiable function on $(\xi_* - \delta, \xi_* + \delta)$ for $\delta \ll 1$. By the Cauchy-Lipschitz theorem, the solution $Q(\xi)$ is necessarily monotonic on $(\xi_* - \delta, \xi_* + \delta)$, which implies that $Q'(\xi_*) < 0$ and a contradiction follows. As for the decreasing property, we assume there exists a first point  $\tilde \xi$ such that $Q'(\tilde\xi)=0$ and $Q'(\xi)<0$ for all $\xi\in (0,\tilde\xi)$, then equation \eqref{equation-Q} gives either $Q(\tilde\xi)=0$ or $Q(\tilde\xi)=\frac 1d$. The first case is not possible because of the strict positivity of $Q$, the second case gives $Q(0)=Q(\tilde\xi)=\frac 1d$ for all $\xi\in (0,\tilde \xi)$, consequently $Q'(\xi)=0$ for all $\xi\in (0,\tilde \xi)$, which is a contradiction. We conclude that $Q$ is positive decreasing and connects $Q(0) = 1/d$ to $Q(\infty) = 0$. By solving $\xi  = \xi(Q)$, we obtain the autonomous ODE
\begin{align*}
\frac{\partial \xi}{\partial Q}  + \frac{Q - 1/2}{dQ(Q - 1/d)} \xi = 0  \quad \mbox{with }\xi(0)=\frac 1d
\end{align*}
whose solution is implicitly given by 
\begin{equation}\label{eq:formQ}
c_\ell \xi^{2\ell} = \frac{1 - d Q}{Q^\ell}, \quad c_\ell \in \Rb_+.
\end{equation}
From the initial condition $Q(0) = 1/d$, we look for an expansion near $\xi = 0$,
\begin{equation} \label{est:TaylorQat0}
  Q(\xi) = \frac{1}{d} -  \frac{c_\ell}{d^{\ell + 1}}  \xi^{2\ell} + \frac{\ell c_\ell^2}{d^{2\ell + 1}} \xi^{4\ell} + \Oc(\xi^{6\ell}) \quad \textup{for} \quad \xi \to 0.
  \end{equation}
As for $\xi$ large, the decay condition $Q(\xi) \to 0$ as $\xi \to \infty$, then from \eqref{eq:formQ} we see that 
\begin{align}
 Q(\xi) \sim c_\ell^{-\frac 1\ell} \xi^{-2}+\Oc (\xi^{-4}) \quad \textup{for} \quad \xi \to \infty. \label{est:QxiInf}
\end{align}
From the relation $F(\xi) = dQ + \xi Q'$, we end up with  
\begin{align*}
F(\xi) =  1 -  \frac{c_\ell (d + 2\ell)}{d^{\ell+1}} \xi^{2\ell} + \frac{c_\ell^2 \ell(d + 4\ell)}{d^{2\ell + 1}} \xi^{4\ell}+ \Oc(\xi^{6\ell}) \quad \textup{as}\;\; \xi \to 0.
\end{align*}
Matching this expansion with \eqref{eq:expinn} yields the value of $c_\ell$:
\begin{equation}\label{eq:con_cell}
 \frac{c_\ell (d + 2\ell)}{d^{\ell+1}} = \frac{(2\alpha)^\ell}{B_\ell}, \quad \textup{hence}, \quad c_\ell = \frac{(2\alpha)^\ell d^{\ell + 1}}{B_\ell (d + 2\ell)}.
\end{equation}
The correct value of $c_\ell$ is used in some algebraic cancellations in improved estimates of the projection onto the null mode in Lemmas \ref{lemm:EEhatys} and \eqref{lemm:finitepart}).\\

We conclude this section by insisting on the fact that the formal derivation of the blowup profile by using the truncated spectral decomposition \eqref{decomp:qtruncated} serves as a formal explanation of why the blowup variable \eqref{def:blowupvar}  appears in dimensions 3 and 4, and would help the readers get an idea on how the blowup dynamic is involved for the system \eqref{sys:KS}. A completely rigorous framework is provided in Sections \ref{sec:linear} and \ref{sec:ControlBs} does not rely on the truncated spectral decomposition \eqref{decomp:qtruncated} at all.

\section{Linearized problem and bootstrap regime} \label{sec:linear}
In this section, the constants $\ell$ and $\alpha$ are fixed as
$$\ell = 3 \;\; \textup{for}\;\; d = 3 \quad \textup{and} \quad \ell = 2 \;\; \textup{for}\;\; d = 4, \quad \alpha = \frac{d-2}{2d} = \frac{1}{2\ell}.$$
We formulate the problem to show that there exist initial data so that the problem \eqref{eq:wys} has a global in time solution that satisfies 
\begin{equation}
\sup_{y \geq 0} \Big| w(y,s) - F\big(y s^{-\frac{1}{2\ell}} \big)  \Big| \to 0 \quad \textup{as} \quad s \to \infty,
\end{equation}
where $F(\xi) = \frac{1}{\xi^{d-1}}\pa_\xi(\xi^d Q(\xi)) = dQ(\xi) + \xi \pa_\xi Q(\xi)$ and $Q(\xi)$ is defined by \eqref{eq:formQ}. 
\subsection{The  partial mass setting}
Since we are working in the radial setting, it is convenient to work in the partial mass to simplify the analysis, namely we introduce 
\begin{align} \label{def:mw}
m_w(y,s) = \int_0^y w(\zeta, s) \zeta^{d-1} d\zeta,  \quad w(y,s) = \frac{\pa_y m_w}{y^{d-1}}, \quad \pa_y \Kc_w = -\frac{m_w}{y^{d-1}}.
\end{align}
We write from \eqref{eq:wys} the equation for $m_w$,
\begin{equation} \label{eq:mw}
\pa_s m_w = \pa_y^2m_w - \frac{d-1}{y}\pa_y m_w  - \frac{1}{2}y \pa_y m_w + \frac{d-2}{2}m_w + \frac{m_w \pa_y m_w}{y^{d-1}}.
\end{equation}
To keep the same scaling invariance of the original solution $w(y,s)$, we further introduce
\begin{align} \label{def:vfrommw}
v(y,s) = \frac{m_w(y,s)}{y^d}, \quad w = d v + y\pa_y v = \frac{1}{y^{d-1}}\pa_y (y^d v),
\end{align}
where we write from \eqref{eq:mw} the equation for  $v$: 
\begin{equation}\label{eq:vys}
\pa_s v =\Delta_{d+2} v - \frac{1}{2}y\pa_y v - v + dv^2 + y v\pa_y v,
\end{equation} 
where $\Delta_{d+2}$ stands for the Laplacian in dimension $d+2$ acting in the radial functions, i.e. 
$$\Delta_{d+2} = \pa_y^2 + \frac{d+1}{y}\pa_y.$$
We have found in the previous section through a formal spectral analysis and matching asymptotic expansions the following approximate blowup profile to \eqref{eq:vys},
$$Q = Q(\xi),  \quad \forall \xi = ys^{-\frac{1}{2\ell}} \geq 0,$$
that solves \eqref{equation-Q} and is defined by \eqref{eq:formQ}. We recall from \eqref{prop:Q} that the profile $Q$ is strictly monotone and positive, 
\begin{equation}\label{est:propQ}
Q(0) = \frac{1}{d}, \quad \lim_{\xi \to \infty}Q(\xi) = 0, \quad  Q(\xi) > 0, \quad Q'(\xi) < 0, \quad \forall \xi > 0. 
\end{equation}
The monotonicity of $Q$ makes the perturbative analysis simpler for the associated linearized problem from \eqref{eq:vys}. In contrary, a linearization from \eqref{eq:wys} around $F$ would make the analysis more complicated in terms of technicalities due to the lack of monotone property of $F$. This is one of the reasons we choose to work with the partial mass setting \eqref{def:vfrommw}. Nevertheless, there is always equivalent between the analysis with \eqref{def:vfrommw} and the one with the original variable $w(y,s)$ up to some technicalities.  

\subsection{Linearization}
A linearization of \eqref{eq:vys} around the profile $Q$, namely 
\begin{equation}
v(y,s) = Q(\xi) + \vep(y,s),
\end{equation}
leads to the linearized problem
\begin{equation} \label{eq:vepys}
\pa_s \vep = \Hs \vep + NL(\vep) + E,
\end{equation}
where $\Hs $ is the second order linear operator
\begin{equation}\label{def:Hs}
\Hs = \Delta_{d+2}  - V_1 y \pa_y   +V_2 ,
\end{equation}
with 
\begin{equation}\label{def:V12}
V_1(\xi)= \frac{1}{2} - Q(\xi), \quad V_2(\xi) =  2dQ - 1 + \xi \pa_\xi Q,
\end{equation}
and $E$ is the generated error, 
\begin{equation} \label{def:E}
E =  \Delta_{d+2} Q - \pa_s Q.
\end{equation}
and $NL$ is the nonlinear quadratic term
\begin{equation}\label{def:NLq}
NL(\vep)  = d\vep^2 + y \vep\pa_y \vep. 
\end{equation}
From the asymptotic behavior of the profile $Q(\xi)$ given in \eqref{est:TaylorQat0} and \eqref{est:QxiInf}, we observe 
$$ V_1(\xi) = \frac{1}{2\ell}  + \Oc(\xi^{2\ell}), \quad V_2(\xi) = 1 + \Oc(\xi^{2\ell}) , \quad \xi \ll 1,$$
and 
$$ V_1(\xi) = \frac{1}{2}  + \Oc(\xi^{-2}), \quad V_2(\xi) = -1 + \Oc(\xi^{-2}) , \quad \xi \gg  1,$$
Thus, the full linearized operator $\Hs$ behaves differently depending on the region:
\begin{equation} \label{def:Hxell}
\Hs \sim \Hs_{\frac{1}{2\ell}} + \textup{Id}, \quad \textup{where}\quad \Hs_{\frac{1}{2\ell}} =  \Delta_{d+2} - \frac{1}{2\ell} y\pa_y, \quad \textup{for}\;\; y \ll s^{-\frac{1}{2\ell}},
\end{equation}
and 
\begin{equation}\label{def:Hs12}
\Hs \sim \Hs_{\frac{1}{2}} - \textup{Id}, \quad \textup{where}\quad \Hs_{\frac 12} = \Delta_{d+2} - \frac{1}{2} y\pa_y, \quad \textup{for}\;\; y \gg s^{-\frac{1}{2\ell}}. 
\end{equation}
We note that in \textit{the outer region} $y \gg s^{-\frac{1}{2\ell}}$, the operator $\Ls$ behaves the same as for the one considered for the classical semilinear heat equation \eqref{eq:NLH}.   However, \textit{the inner region} $y \ll s^{-\frac{1}{2\ell}}$ and \textit{the intermediate region} $y \sim s^{-\frac{1}{2\ell}}$, the operator behaves differently in comparison with what is known in the analysis for \eqref{eq:NLH}. To our knowledge, there is no a complete spectral theory for the full linear operator $\Hs$.  We thus use a different approach to avoid this missed piece in the analysis, especially in the intermediate region. 

We recall here the spectral properties of $\Hs_{\frac{1}{2\ell}}$, which play an important role in the analysis when the inner region is concerned. The linear operator $\Hs_{\frac{1}{2\ell}}: H^2_{\rho}(\Rb_+) \to L^2_{\rho}(\Rb_+)$ is self-adjoint, where  the weight function 
$$\rho(y) = e^{- \frac{|y|^2}{4\ell}} y^{d+1}.$$
In particular, we have from \eqref{eq:phi2n} and \eqref{def:vfrommw},
\begin{equation}
\Hs_{\frac{1}{2\ell}} \varphi_{2n} = - \frac{n}{\ell} \varphi_{2n}, \quad n \in \mathbb{N},
\end{equation}
where the following relation holds
\begin{equation}
\varphi_{2n}(y) = \frac{1}{y^{d}}\int_0^y \phi_{2n}(\zeta) \zeta^{d-1}d\zeta,
\end{equation}
with $\phi_{2n}$ being defined explicitly in \eqref{def:phi2ell}. We use the even index $2n$ instead of $n$ to infer that we are in the radial setting and the eigenfunctions are  only polynomials of even degrees. Note that $\varphi_{2n}$ has the same form as $\phi_{2n}$. We also have the orthogonality 
\begin{equation}
\int_0^\infty \varphi_{2n}(y) \varphi_{2m}(y) \rho(y) dy = c_{n} \delta_{n,m},
\end{equation}
and the family of the eigenfunctions $\{ \varphi_{2n}\}_{n \in \mathbb{N}}$ forms a complete orthogonal basis in $L^2_\rho(\Rb_+)$ in the sense that for any $g \in L^2_\rho(\Rb_+)$ we can decompose it as 
\begin{equation}
g(y) = \sum_{n \in \mathbb{N}} g_n \varphi_{2n}(y), \quad g_n = \langle g, \varphi_{2n}\rangle_\rho = \int_0^\infty g \varphi_{2n} \rho dy. 
\end{equation}
\begin{remark}[A conjecture on the spectral properties of $\Hs$] For $d = 3,4$, $s > 0$, let $\ell = \frac{d}{d-2}$ and $\omega[s](y) = y^{d+1}e^{-\int yV_1(y,s)dy}$. We conjecture that the linearized operator acting on radial functions $\Hs[s]: H^2_{\omega[s]} \to L^2_{\omega[s]}$ is self-adjoint and its spectrum consists of $\ell$ positive eigenvalues, a zero eigenvalue and infinitely many negative ones. Let $\big\{\phi_k[s]\big\}_{0 \leq k \leq \ell}$ be the first $\ell+1$ radial eigenfunctions corresponding to the $\ell$ positive and zero eigenvalues, we have the spectral gap 
$$ \langle \Hs v, v\rangle_{L^2_{\omega[s]}} \leq -\delta \|v\|^2_{L^2_{\omega[s]}} + C\sum_{k=0}^\ell \langle u, \phi_k[s]\rangle_{L^2_{\omega[s]}} \quad \textup{for some}\;\; \delta, C > 0.$$
We remark that in the blowup variable $\xi = \frac{y}{s^\frac{1}{2\ell}}$, the operator $\Hs[s]$ acting on functions $f = f(\xi)$ becomes $$\Hs[s]f(\xi) = s^{-\frac{1}{\ell}} \Delta_{d + 2} f(\xi) + \As f(\xi), \quad \textup{where}\;\; \As = - V_1(\xi) \xi \partial_\xi + V_2(\xi).$$
This observation yields $\Hs[s] \sim \As$ as $s \to \infty$, from which we expect that a perturbative analysis would yield a complete spectral description for $\Hs[s]$ by studying $\As$. As for the linear transport operator $\As$, we conjecture that $\As$ acting on smooth radial functions admits the point spectrum $\Big\{\frac{\ell - k}{\ell}, k \in \mathbb{N}\Big\}$. The zero eigenvalue can be verified by a scaling argument as we note that the profile equation
\begin{equation}\label{eq:Qequ}
 -\frac{1}{2}\xi \pa_\xi Q - Q + dQ^2 + \xi Q \pa_\xi Q = 0,
 \end{equation}
is invariant under the dilation $Q_\gamma(\xi) = Q(\gamma \xi)$ for any $\gamma > 0$. By plugging $Q_\gamma$ into \eqref{eq:Qequ} and  differentiating it with respect to $\gamma$ near $\gamma = 1$, we get $\As \big(\xi \pa_\xi Q\big) = 0$. Let $\varphi_\ell(\xi) = \xi \partial_\xi Q$, then $\As \varphi_\ell(\xi) = 0$. We check that the function $\varphi_k(\xi) = \xi^{2k - 2\ell} \varphi_\ell(\xi)$ satisfies 
$$ \As \varphi_k(\xi) = 2(\ell-k)V_1(\xi) \varphi_k(\xi) \sim \frac{\ell - k}{\ell} \varphi_k(\xi),$$
where we used the asymptotic behavior $V_1(\xi)  = \frac{1}{2\ell} + \Oc(\xi^{2\ell})$ as $\xi \to 0$. 
\end{remark}

\subsection{Bootstrap regime}
Our aim is to construct a global in time solution $\vep(y,s)$ to \eqref{eq:vepys} that satisfies 
\begin{equation}
\|\vep(s)\|_{L^\infty(\Rb_+)} \to 0 \quad \textup{as} \quad s \to \infty. 
\end{equation} 
This requirement implies that the dynamics of \eqref{eq:vepys} mainly relies on the linear part $\Hs$ since the nonlinear term is roughly quadratic. The construction is based on the following observation:

\paragraph{- the outer region $y \gg s^\frac{1}{2\ell}$ ($\xi \gg 1$):} thanks to the decay of $Q(\xi)$, the linear part $\Hs \sim \Hs_{\frac 12} - \textup{Id}$ has a fully negative spectrum. Thus, we can control the solution in this region without difficulties. In particular, let $K \gg 1$ be a large fixed constant and define
\begin{equation}\label{def:vep_out}
\vep^\out(y,s) = \vep(y,s) \big(1 - \chi_{_{K}}(\xi)\big),
\end{equation}
where
\begin{equation}\label{def:chiK}
\chi_{_K}(\xi) = \chi_{_0}\Big( \frac{\xi}{ K}\Big),
\end{equation}
and $\chi_{_0}$ is the cut-off function 
$$\chi_{_0} \in \Cc^\infty\big(\Rb_+, [0,1]\big), \quad \chi_{_0}(\xi) = 1 \; \textup{if}\;\xi \in [0,1] \quad \textup{and} \quad \chi_{_0}(\xi) = 0\; \textup{if}\;\xi \geq 2.$$ 
From \eqref{eq:vepys}, we write the equation for $\vep^\out(y,s)$,
\begin{equation}\label{eq:vepout}
\pa_s \vep^\out = \big(\Hs_\frac 12 - \textup{Id} \big) \vep^\out + \big(1 - \chi_{_{K}} \big)\big[ Q y\pa_y \vep + (2dQ + \xi \pa_\xi Q) \vep + NL(\vep) + E\big] + \Ec^{bd}(\vep),
\end{equation}
and $\Ec^{bd}$ is the boundary term due to the cut-off defined by  
\begin{equation}\label{def:EM}
\Ec^{bd}(\vep) = \Big(-\pa_s \chi_{_K} + \Delta_{d+2} \chi_{_K}  - \frac 1{2\ell} y \pa_y \chi_{_K}\Big) \vep+ 2\pa_y \chi_{_K}\pa_y \vep.
\end{equation} 
We need here the information of $\vep$ at the boundary $K s^{\frac{1}{2\ell}} \leq y \leq 2K s^\frac{1}{2\ell}$ ($K \leq \xi \leq 2K$) that we retrieve from the estimate of $\vep$ in the intermediate region to close the estimate for $\vep^\out$. \\

\medskip

\paragraph{- the intermediate region $y \sim s^\frac{1}{2\ell}$ ($\xi \sim 1$):} As mentioned earlier, there is a lack of a complete description of spectral property of $\Hs$, we are not able to use the semigroup estimate as for the nonlinear heat \eqref{eq:NLH} (see for example, \cite{MZdm97}, \cite{NZens17}).  Thanks to the monotone property of $Q$ and the dissipation, we can achieve the control of $\vep$ in this region through a standard energy estimate from \eqref{eq:vepys}. To obtain a small enough estimate, we need to refine the approximate solution by introducing 
\begin{equation}\label{def:PsiApp}
\Psi(y,s) = Q(\xi) + \hat{\Psi}(y,s), \quad \hat{\Psi}(y,s) =  - \frac{1}{B_\ell s} \Big(  \varphi_{2\ell}(y) - \frac{(2\alpha y^2)^\ell}{2\ell + d} \Big)\chi_{_0}(\xi),
\end{equation}
where $B_\ell$ is the constant defined in \eqref{def:valueB2} and \eqref{def:valueB3} and $2\alpha = \frac{1}{\ell}$. Recall from \eqref{eq:expinn} and \eqref{def:vfrommw}, we have an equivalent inner expansion of $v(y,s)$ in $L^2_\rho$:
\begin{equation}\label{eq:InnerExpv}
v(y,s) = \frac{1}{d} - \frac{1}{B_\ell s} \varphi_{2\ell}(y) + \cdots,
\end{equation}
and  from \eqref{est:TaylorQat0} and \eqref{eq:con_cell}, 
\begin{equation}\label{eq:Qat0}
Q(\xi) = \frac{1}{d} - \frac{1}{B_\ell(2\ell + d)} (2\alpha \xi^2)^\ell + \frac{\ell d }{B_\ell^2(2\ell + d)^2} (2\alpha \xi^2)^{2\ell} + \Oc(\xi^{6\ell}) \quad \textup{as} \quad \xi \ll 1.
\end{equation} 
Hence, we have for $ y < s^\frac{1}{2\ell}$, 
\begin{align}
\Psi(y,s) &= \frac{1}{d} - \frac{\varphi_{2\ell}(y) }{B_\ell s} +  \frac{\ell d }{B_\ell^2(2\ell + d)^2}  \frac{(2\alpha y^2)^{2\ell}}{s^2} +   \Oc\Big( \frac{\langle y\rangle^{6\ell}}{ s^{3}} \Big), \label{eq:Psiat0}
\end{align}
which agrees with the expansion \eqref{eq:InnerExpv}.

We then linearize
\begin{align}
v(y,s) =  \Psi(y,s) + \hat \vep(y,s),\label{def:vephat}
\end{align}
and write from \eqref{eq:vepys} and the relation $\vep = \hat{\Psi} + \hat \vep$, 
\begin{equation}\label{eq:vephat}
\pa_s \hat \vep = \Hs \hat \vep + \hat\Vc \hat\vep  + NL(\hat \vep) + \hat E,
\end{equation}
where $\Hs$, $NL$ are defined in \eqref{def:Hs}, \eqref{def:NLq}, and $\hat \Vc$ is small linear operator
\begin{equation}\label{def:hatVc}
\hat \Vc \hat \vep = -\hat \Psi y \pa_y \hat \vep  + (2d \hat \Psi + y\pa_y \hat \Psi) \hat \vep, 
\end{equation}
and $\hat E$ is the generated error
\begin{equation}\label{def:Ehat}
\hat E(y,s) = -\pa_s (Q(\xi) + \hat \Psi) + \Delta_{d+2}Q(\xi) + \Hs \hat \Psi + NL(\hat \Psi).
\end{equation}
We introduce the following norm: for a fixed large constant $K\gg 1$,
\begin{equation}\label{def:normvepmid}
 \|\hat \vep(s)\|_{\flat}^2 = \int_{0}^\infty (1 - \chi_{_K}(y)) \Big(\frac{|\hat \vep|^2}{|y|^{4\ell + 2}}\Big)\frac{dy}{y}.
\end{equation}
In fact, we can replace the power $4\ell + 2$ by any integer number $2k$ with $k \geq 2\ell + 1$, so that after some integration by parts, we get an estimate of the form 
$$\frac{d}{ds} \|\hat \vep(s)\|_{\flat}^2 \leq -\delta(k) \|\hat \vep(s)\|_{\flat}^2 +   \|\hat E(s)\|_{\flat}^2 + \textup{"boundary terms $y \sim K$"}, $$
where $\delta(k)$ is strictly positive for $k \geq 2\ell + 1$. Due to the cut-off $(1 - \chi_{_K}(y))$, we need information of $\hat \vep$ for $K \leq y \leq 2K$ to estimate the boundary term and to complete the estimate of $\|\hat \vep(s)\|_{\flat}$. This information is retrieved from the control in the inner region. \\

\medskip

\paragraph{- the inner region $y \ll s^\frac{1}{2\ell}$ ($\xi \ll 1$):} The linearized operator $\Hs$ behaves like $\Hs_{\frac{1}{2\ell}} + \textup{Id}$ that has $\ell - 1$ positive eigenvalues, a zero mode and infinite many negative ones.  We need further refinement to achieve the control of $\hat \vep$ in this region. More precisely, we decompose
\begin{align}
\hat \vep(y,s) = \Psi(y,s) + \hat{\vep}_\natural(y,s) + \tilde \vep(y,s), \label{decomp:vepys}
\end{align}
where 
\begin{equation}\label{def:vepnatural}
\hat{\vep}_\natural(y,s) = \sum_{k = 0}^{2\ell - 1} \hat \vep_k(s) \varphi_{2k}(y), \quad \hat \vep_k(s) = \langle \hat \vep, \varphi_{2k}\rangle_\rho. \
\end{equation}
The introduction of $\hat \vep_\natural$ is to obtain a spectral gap estimate once we perform $L^2_\rho$ estimate for $\tilde \vep$. In fact, we have the orthogonality condition
\begin{equation}\label{eq:orhtogonalvep}
\langle \tilde \vep, \varphi_{2k}\rangle_\rho = 0 \quad \textup{for} \quad k = 0, \cdots, \ell,
\end{equation}
 where $\varphi_{2k}$ is the eigenfunction associated to the linear operator $\Hs_{\frac{1}{2\ell}}$, from which we have the estimate
 \begin{equation}\label{est:spectralgap}
 \langle \Hs_{\frac{1}{2\ell}} \tilde{ \vep}, \tilde{\vep} \rangle_\rho + \langle \tilde{\vep}, \tilde{\vep}\rangle_\rho \leq  - \langle \tilde{\vep}, \tilde{\vep}\rangle_\rho. 
 \end{equation}
From \eqref{eq:vepys}, we write the equation for $\tilde{\vep}$, 
\begin{equation}\label{eq:veptil}
\pa_s \tilde \vep = \big(\Hs_\frac{1}{2\ell} + \textup{Id}\big) \tilde \vep + \tilde \Vc \tilde \vep + NL(\tilde{\vep}) + \tilde E(y,s),
\end{equation}
where $\Hs$ is the linearized operator around $Q$ introduced in \eqref{def:Hs}, $\Vc$ is a small first order linear term, 
\begin{equation}\label{def:Vcqtil}
\tilde \Vc   = - \tilde{V}_1 y\pa_y + \tilde V_2 ,
\end{equation}
with 
\begin{equation} \label{def:V12til}
\tilde{V}_1 = \frac{1}{d} - \Psi - \hat{\vep}_\natural, \quad \tilde V_2 = 2d \Psi  -2 + y \pa_y \Psi + 2d \hat{\vep}_\natural + y\pa_y \hat{\vep}_\natural,
\end{equation}
the nonlinear term $NL(\tilde{\vep})$ is defined by \eqref{def:NLq} and $E$ is the total error term, 
\begin{equation} \label{def:Etil}
\tilde E(y,s) =  -\pa_s (\Psi + \hat{\vep}_\natural) + \Delta_{d+2}Q(\xi) +  \Hs(\hat \Psi+ \hat{\vep}_\natural) + NL( \hat \Psi + \hat{\vep}_\natural).
\end{equation}

\noindent We now define the bootstrap regime to fully control the solution to \eqref{eq:vepys}. 
\begin{definition}[Bootstrap regime]\label{Definition-shrinking -set} Let $s > 1$ and $A > 1$, we define  $\Sc_A(s)$ the set of all functions $\vep(s) \in W^{1,\infty}(\Rb_+)$ such that 
\begin{equation}\label{est:vepk}
  |\hat \vep_k(s)| \leq A s^{-2} \quad \textup{for} \quad 0 \leq k \ne \ell \leq 2\ell- 1, 
\end{equation}
\begin{equation}\label{est:vepell}
 |\hat \vep_\ell(s)| \leq A^2 s^{-2}\log s,
\end{equation}
\begin{equation} \label{est:vepL2rho}
\| \tilde{\vep}(s) \|_{L^2_\rho(\Rb_+)} \leq A s^{-3},
\end{equation}
\begin{equation}\label{est:vepmid}
j =0,1,2, \quad \|(y\pa_y)^j \hat \vep(s)\|_{\flat} \leq A^{1+j} s^{-1 - \frac{3}{2\ell}},
\end{equation}
\begin{equation}\label{est:vepout}
j =0,1,\quad \|(y \pa_y)^j\vep^\out(s)\|_{L^\infty(\Rb_+)} \leq A^{4+j} s^{-\frac{1}{\ell}}, \quad  \|y\vep^\out(s)\|_{L^\infty(\Rb_+)} \leq A^{4} s^{-\frac{1}{2\ell}},
\end{equation}
where $\hat \vep$, $\tilde \vep$, $\vep^\out$ and $\|\cdot\|_{\flat}$  are introduced in \eqref{decomp:vepys}, \eqref{def:vep_out} and \eqref{def:normvepmid}.
\end{definition}
\begin{remark}[Order of estimates] The bootstrap estimates defined in the shrinking set $\Sc_A$ shows the priorities in order to achieve the control of $\vep$ in the whole $\Rb_+$ as follows: we first obtain $L^2_\rho$-estimate for $\tilde{\vep}$ which directly gives $L^2_{loc}$-estimate, then a standard parabolic regularity yields $L^\infty(y \lesssim K)$ bound for $\hat \vep$ for any $K > 0$; this $L^\infty_{loc}$-estimate is then used in the energy estimate of $\|\hat \vep(s)\|_{\flat}$ to bound boundary terms (having the support $y \in [K, 2K]$) due to the cut-off $\chi_{_K}(y)$ (see \eqref{def:normvepmid}). A parabolic regularity type argument yields a similar estimate for $\|y\pa_y\hat \vep(s)\|_{\flat}$, from which and Sobolev we get $L^\infty$-estimate for $y \lesssim K s^{\frac{1}{s^{2\ell}}}$. This $L^\infty$-bound for $y \sim K s^{\frac{1}{s^{2\ell}}}$ enters the estimate of $\|\vep^\out\|_{L^{\infty}}$ due to the cut-off $\chi_{_K}(\xi)$. We thus can obtain an $L^\infty$ bound for $\vep$ in the whole $\Rb_+$. Since the nonlinearity is quadratic, we just need a rough bound of $\|\vep\|_{L^\infty(\Rb)}$ to handle this nonlinear term in all estimates. 
\end{remark}

\subsection{Existence of solutions in the bootstrap regime}
The strategy is to show that if we start with initial data $\vep(s_0) \in \Sc_A(s_0)$, then the corresponding solution $\vep(s)$ to the equation \eqref{eq:vepys} stays in $\Sc_A(s)$ for all $s \geq s_0$. In particular, we consider the  initial data of the form
\begin{equation}\label{def:intitialdata_q}
\vep(y,s_0) = \hat{\Psi}(y,s_0) + \hat \vep(y,s_0), \quad \hat \vep(y,s_0) \equiv \hat \psi[\mathbf{d},A, s_0](y)= \frac{A}{s_0^2}\left(\sum_{i=0}^{\ell-1}d_i \varphi_{2i}\right) \chi_{_0}(\xi) 
\end{equation}
where $\hat{\Psi}$ is defined in \eqref{def:PsiApp}, $\varphi_{2i}$'s are the eigenfunctions of $\Hs_\frac{1}{2\ell}$, $\chi_{_0}$ is the cut-off function introduced right after \eqref{def:chiK}, 
$$ s_0 \gg 1, \;\; A \gg 1, \; \; \mathbf{d} = (d_0, \cdots, d_{\ell - 1}) \in B_1(\Rb^{\ell}),$$
are real parameters to be determined for which the corresponding solution $\vep(s)$ is trapped in $\Sc_A(s)$ for all $s \geq s_0$. Precisely, we aim at proving the following  proposition which is the central of our construction leading to the conclusion of Theorem \ref{theo:1}. 

\begin{proposition}[Existence of solutions in $\Sc_A$] \label{prop:2} There are $s_0 \gg 1 , A \gg 1$ and $\mathbf{d} \in B_1(\Rb^\ell)$ such that the solution $\vep(s)$ to \eqref{eq:vepys} with the initial data $\vep(s_0) = \psi[\mathbf{d}, A, s_0]$ defined in \eqref{def:intitialdata_q} is trapped in $\Sc_A(s)$ for all $s \geq s_0$. 
\end{proposition}
\begin{proof} By the definition of $\hat \psi[\mathbf{d}, A, s_0]$ and the projection of $\hat\psi[\mathbf{d}, A, s_0]$ onto $\varphi_{2k}$, we obtain by a direct computation and the exponential decay of the weight function $\rho$,
\begin{equation} \label{est:Psis0ell}
\hat \psi_k = \frac{A d_k}{s_0^2} +\Oc(s_0^{-2}e^{-\kappa s_0^{1/\ell}}) \quad \textup{for} \;\; k = 1, \cdots, \ell - 1, 
\end{equation}
and 
$$  |\hat \psi_k | = \Oc( A s_0^{-2}e^{-\kappa s_0^{1/\ell}}) \quad \textup{for}\;\; k \geq \ell,$$
for some $\kappa > 0$, and  
\begin{align*}
|\tilde \psi(y)| &= \big| \hat \psi(y) - \sum_{k = 0}^{2\ell-1} \hat \psi_k \varphi_{2k}(y)\big| = \Big|\sum_{k = 0}^{\ell - 1} \hat \psi_k \varphi_{2k}(y) \big(1 - \chi_{_0}(\xi)\big) - \sum_{k = \ell}^{2\ell - 1}\hat \psi_k \varphi_{2k}(y) \Big|\\
& \lesssim \frac{A}{s_0^2} \langle y \rangle^{2\ell -2} \mathbf{1}_{\{\xi \geq 1\}} +  A s_0^{-2}e^{-\kappa s_0^{1/\ell}} \yj^{4\ell - 2}, \quad \forall y > 0.  
\end{align*}
This yields the bounds 
$$ \|\tilde \psi\|_{L^2_\rho} \lesssim A s_0^{-2}e^{-\kappa s_0^{1/\ell}}, \quad \sum_{j = 0}^2 \| (y\pa_y)^j \hat \psi\|_{\flat} \lesssim A s_0^{-2}e^{-\kappa s_0^{1/\ell}}. $$
By the definition \eqref{def:vep_out}, we have by $\chi_{_0}(\xi) \big(1 - \chi_{_K}(\xi) \big) = 0$ for $K \geq 2$, thus, $\vep^\out(s_0) \equiv 0$. We then conclude that for $A \gg 1$ and $s_0 \gg 1$, the initial data $\vep(y,s_0) \in \Sc_A(s_0)$ with strictly inequalities, except for the first $\ell$ components $\hat \psi_k$ with $k = 0, \cdots, \ell - 1$. \\

\noindent From the local Cauchy problem of \eqref{sys:KS} in the radial setting  in $L^\infty(\Rb^d)$, for each initial data $ \vep( s_0) = \tilde{\Psi}(s_0) + \hat \psi_{\mathbf{d}, A, s_0} \in \Sc_A(s_0)$, there is a unique solution $\vep(s) \in \Sc_A(s)$ for $s \in [s_0, s_*)$. If $s_* = +\infty$, we are done. If $s_* < \infty$, we claim that $\vep(s_*)$ touches the boundary $\partial \Sc_A(s_*)$ only for the first $\ell$ components $\hat \vep_k(s_*)$ with $k = 0, \cdots, \ell-1$. Here, $\partial \Sc_A(s_*)$ means that the bootstrap estimates in Definition \ref{Definition-shrinking -set} hold with the equal signs. In particular, we claim the following:
\begin{proposition}[Reduction to finite dimensional problem] \label{prop:3} For $A \gg 1$, $s_0 = s_0(A) \gg 1$, there exists $\mathbf{d} = (d_0, \cdots, d_{\ell-1}) \in B_1(\Rb^{\ell})$ such that if the solution $\vep(s)$ to \eqref{eq:vepys} with the initial data $ \vep(y,s_0) = \tilde{\Psi}(y,s_0) + \hat\psi[\mathbf{d}, A, s_0](y)$ defined as in \eqref{def:intitialdata_q} satisfies $\vep(s) \in \Sc_A(s)$ for $s \in [s_0, s_*]$ and $\vep(s_*) \in \partial \Sc_A(s_*)$, then it holds\\
\begin{equation} \label{est:qks*}
\big(\hat \vep_0, \cdots, \hat \vep_{\ell-1}\big)(s_*)  \in \partial \Big(\Big[ -\frac{A}{s_*^2}, \frac{A}{s_*^2}\Big]\Big)^{\ell}.
\end{equation}
Moreover, we have 
\begin{equation}\label{transversecross}
\frac{d}{ds} \hat \vep_k^2(s_*) > 0  \quad \textup{for}\;\; k = 0, \cdots, \ell-1.
\end{equation}
\end{proposition}
\noindent The proof of Proposition \ref{prop:3} is technically long and is left to Section \ref{sec:proofofmainthm} after having derived necessary dynamics of the finite dimensional part (Lemma \ref{lemm:finitepart}), the $L^2_\rho$-energy estimate (Lemma \ref{lemm:L2rho}), the energy estimate in the intermediate and outer regions (Lemmas \ref{lemm:mid} and \ref{lemm:outer}). Let us assume Proposition \ref{prop:3} and complete the proof of Proposition \ref{prop:2}. The idea is to use a Brouwer fixed point argument to obtain a contradiction if $s_* < +\infty$. Let $\mathcal{B}_1$ be the unit ball in $\mathbb{R}^{\ell}$ and define the map 
\begin{align*}
\Theta: D(\Theta) \subset \mathcal{B}_1 & \to \partial \mathcal{B}_1,\\
\frac{s_0^2}{A}\big(\hat \vep_0, \cdots, \hat \vep_{\ell-1} \big)(s_0) & \mapsto  \frac{s_*^2}{A}\big(\hat \vep_0, \cdots, \hat \vep_{\ell-1} \big)(s_*),
\end{align*}
where 
$$D(\Theta) = \Big\{\frac{s_0^2}{A}\big(\hat \vep_0, \cdots, \hat \vep_{\ell-1} \big)(s_0) \in \Rb^\ell: \; \frac{s_0^2}{A}\big(\hat \vep_0, \cdots, \hat \vep_{\ell-1} \big)(s_0) \in \mathcal{B}_1, \quad and \;\; s_* < +\infty \Big\}.$$
From \eqref{est:Psis0ell}, we can take $(d_0, \cdots, d_{\ell -1}) \in \Rb^\ell$ such that $\frac{s_0^2}{A}\big(\hat \vep_0, \cdots, \hat \vep_{\ell-1} \big)(s_0) \in \partial \mathcal{B}_1$. Then we have by the transverse crossing \eqref{transversecross}, it holds $\sum_{k=0}^{\ell-1} \frac{s^2}A |\hat \vep_k (s)|^2 > 1$ for any $s > s_0$. From \eqref{est:qks*}, the only possibility is $s_* = s_0$, hence, $\Theta$ is the identity map on $\partial\mathcal{B}_1$. Since the unstable modes $(\hat \vep_0, \cdots, \hat \vep_{\ell - 1})(s)$ depend continuously with respect to the initial data involving the parameters $(d_0, \cdots, d_{\ell - 1})$, and the vector field $(\hat \vep_0, \cdots, \hat \vep_{\ell - 1})(s)$ satisfies the strictly outgoing property \eqref{transversecross}, hence, $D(\Theta)$ is open and $\Theta$ is continuous on $D(\Theta)$. Since $\Theta$ is continuous in $\mathcal{B}_1$ and the identity map on the boundary, then it implies that $\partial \mathcal{B}_1$ is a retract on $\mathcal{B}_1$, which is a contradiction to Brouwer's fixed point theorem \cite{Bma12}. Therefore, $s_* = \infty$, which concludes the proof of Proposition \ref{prop:2}, assuming Proposition \ref{prop:3} whose proof is given in Section \ref{sec:proofofmainthm}.
\end{proof}

\section{Control the solution in the bootstrap regime} \label{sec:ControlBs}

\subsection{Properties of the shrinking set}
\noindent We claim the following. 
\begin{lemma}[Properties of the shrinking set] Let $A \gg 1$ and $s \geq s_0 \gg 1$ and $\vep(s) \in \Sc_A(s)$ be a solution to \eqref{eq:vepys}. We have \\
i) (Local $L^\infty$-estimate) For all $M > 0$ and $j = 0,1$, 
\begin{equation}\label{est:vepLinfloc}
\|\pa_y^j \tilde \vep(s)\|_{L^\infty(y \leq M)} \lesssim C(M) As^{-3}, \quad  \|\pa_y^j \hat \vep(s)\|_{L^\infty(y \leq M)} \lesssim C(M) A^2 s^{-2}|\log s|.  
\end{equation}
ii) (Pointwise estimate)
\begin{equation}\label{est:vephatpointswise}
\forall y > 0, \quad |\hat \vep(y,s)|  + |y\pa_y \hat \vep(y,s)| \lesssim A^3 s^{-1 - \frac{3}{2\ell}} \yj^{2\ell + 1}. 
\end{equation}
iii) (Global $L^\infty$-estimate)
\begin{equation} \label{est:vepLinfbound}
\| \hat \vep(s)\|_{L^\infty(\Rb_+)} + \| (y\pa_y) \hat \vep(s)\|_{L^\infty(\Rb_+)} \lesssim A^5s^{-\frac{1}{\ell}}. 
\end{equation}
\end{lemma}
\begin{proof} (i) From the $L^2_\rho$ bound \eqref{est:vepL2rho} of $\tilde{\vep}$, we get 
$$\| \tilde{\vep}(s)\|_{L^2(y \leq 2M)} \lesssim e^{\frac{M^2}{2\ell}}\| \tilde \vep(s)\|_{L^2_\rho} \lesssim C(M) A s^{-3}. $$
A standard parabolic regularity then yields the estimate 
$$\| \tilde{\vep}(s)\|_{L^{\infty}(y \leq M)} + \| \pa_y\tilde{\vep}(s)\|_{L^{\infty}(y \leq M)} \lesssim C(M)A s^{-3}.$$
We recall from the decomposition \eqref{decomp:vepys},
\begin{equation}\label{def:vephat}
\hat \vep(y,s) = \sum_{k = 0}^{2\ell - 1} \hat \vep_k(s) \varphi_{2k}(y) + \tilde{\vep}(y,s).
\end{equation}
Using the bootstrap bounds \eqref{est:vepk} and the $L^\infty$ local bound of $\tilde \vep$ yields the desired estimates. \\
(ii) We first claim that if $\int_{y \geq 1} (u^2 + |y\pa_y u|^2) y^{-1}dy < +\infty$, then 
\begin{equation}\label{est:sobolev}
\| u \|_{L^\infty(y \geq 1)}  \lesssim \int_{y \geq 1} (u^2 + |y\pa_y u|^2) y^{-1}dy. 
\end{equation}
By making a change of variable $v(z) = u(M z)$,  we write for $M \in \{2^j\}_{j = 0}^\infty$,
\begin{align*}
\| u\|^2_{L^\infty([M, 2M])} &= \| v\|^2_{L^\infty([1,2])} \lesssim \int_1^2 v^2(z) dz + \int_1^2 |\pa_z v(z)|^2 dz\\
& \lesssim \int_M^{2M} v^2(y) \frac{dy}{M} + \int_{M}^{2M} M|\pa_y v(y)|^2 dy \\
& \lesssim \int_M^{2M} v^2(y) y^{-1} dy + \int_{M}^{2M} |y \pa_y v(y)|^2 y^{-1} dy. 
\end{align*}
Taking the supremum in $j$ on the left-hand side of the above estimate and summing all the contributions from the right-hand side yields the desired inequality \eqref{est:sobolev}.
We then apply \eqref{est:sobolev} with $u = \frac{\hat \vep}{y^{2\ell + 1}}$ and $u = \frac{y\pa_y\hat \vep}{ y^{2\ell + 1}} $ to obtain from \eqref{est:vepmid}
$$ \forall y \geq 1, \quad | \hat \vep(y,s)| + |y\pa_y \hat \vep(y,s)| \lesssim A^{3} s^{-1 - \frac{3}{2\ell}} \yj^{2\ell + 1},$$
from which and \eqref{est:vepLinfloc}, we obtain \eqref{est:vephatpointswise}. \\
(iii) The estimate \eqref{est:vepLinfbound} follows from \eqref{est:vephatpointswise} and the bootstrap bound \eqref{est:vepout}. 
\end{proof}

\subsection{Decomposition of the error}
\noindent We claim the following. 
\begin{lemma}[Estimate of the generated error]  \label{lemm:EEhatys} We have\\
(i) \textup{($L^\infty$-bound of $E$ and $\hat E$)}
\begin{equation}\label{est:ELinf} 
\|E(s)\|_{L^\infty(\Rb_+)} + \|y\pa_y E(s)\|_{L^\infty(\Rb_+)}+ \|\hat E(s)\|_{L^\infty(\Rb_+)} + \|y\pa_y \hat E(s)\|_{L^\infty(\Rb_+)} \lesssim s^{-\frac 1{\ell}}.
\end{equation}
(ii) \textup{(Decomposition of $\hat E$)} 
\begin{equation}\label{dec:Ehat}
\hat E(y,s) = \sum_{k = 0}^{2\ell - 1} \hat E_k(s) \varphi_{2k}(y) + \hat R(y,s),
\end{equation}
where 
\begin{equation}\label{est:Ehatkandell}
\sum_{k = 0, k \ne \ell}^{2\ell-1}|\hat E_k(s)| \lesssim s^{-2}, \quad  \quad |\hat E_\ell(s)| \lesssim s^{-3}, \quad \|\hat R(s)\|_{L^2_\rho} \lesssim s^{-3},
\end{equation}
and 
\begin{equation}\label{est:Rhatpointwise}
\forall y \lesssim s^\frac{1}{2\ell}, \quad  \sum_{j = 0}^2 |(\yj \pa_y)^j\hat R(y,s)| \lesssim s^{-3}\yj^{6\ell - 2}.
\end{equation}
In particular, we have
\begin{equation}\label{est:Ehatmid}
\sum_{j = 0}^2\int_{1}^\infty \frac{|(y\pa_y)^j \hat E(y,s)|^2}{y^{4\ell +2}} \frac{dy}{y} \lesssim s^{-2 - \frac{3}{\ell}}. 
\end{equation}
\end{lemma}

\begin{remark} The improved estimate for $\hat E_\ell$ reaching the order $s^{-3}$ comes from a crucial algebraic cancellation to remove the term of order $s^{-2}$ thanks to the precise choice of \eqref{eq:con_cell}. 
\end{remark}

\begin{proof} (i) From \eqref{def:E}, we have the estimate for all $y \geq 0,$
\begin{align*}
 |E(y,s)| &=  |E(\xi, s)| =  \big|-\pa_s Q(\xi) + \Delta_{d+2}Q(\xi) \big|\\
 & \lesssim s^{-1}\Big|\xi Q'(\xi)\Big|  + s^{-\frac{1}{\ell}} \Big|Q''(\xi) + \xi^{-1}Q'(\xi) \Big| \lesssim s^{-\frac{1}{\ell}},
\end{align*}
and 
\begin{align*}
 |y\pa_yE(y,s)| = |\xi \pa_\xi E(\xi, s)| = s^{-1}\Big| (\xi \pa_\xi)^2Q(\xi)\Big|  + s^{-\frac{1}{\ell}} \Big| \xi \pa_\xi Q''(\xi) + \xi \pa_\xi (\xi^{-1}Q'(\xi))\Big| \lesssim s^{-\frac{1}{\ell}}.
\end{align*}
As for $\hat E$, we have by the definitions of $\varphi_{2\ell}$ and $\hat \Psi$,
$$j = 0, \cdots, 3, \quad  |(\yj\pa_y)^j\hat \Psi(y,s)| \lesssim s^{-\frac{1}{\ell}} , \quad |\pa_s (\yj \pa_y)^j\hat \Psi(y,s)| \lesssim s^{-1 - \frac{1}{\ell}}, \quad \forall y \geq 0. $$
Then, we have from \eqref{def:Ehat},  \eqref{def:Hs} and the bound $Q + |\xi \pa_\xi Q| \lesssim 1$, 
\begin{align*}
|\hat E(y,s)| &= | E(y,s) - \pa_s \hat \Psi + \Hs \hat \Psi + NL(\hat \Psi)| \\
& \lesssim |E(y,s)| + |\pa_s \hat \Psi| + |\Delta_{d+2} \hat \Psi| + |y\pa_y \hat \Psi| + |\hat \Psi| + |y \pa_y\hat \Psi \hat \Psi | + |\hat \Psi|^2 \lesssim s^{-\frac{1}{\ell}},
\end{align*}
and a similar bound for $|y\pa_y \hat E|$, which concludes the proof of $(i)$.\\

(ii) From \eqref{def:Ehat}, \eqref{def:Hs} and \eqref{def:Hxell}, we rewrite $\hat E$ as
\begin{align*}
\hat E(y,s) &= -\pa_s \Psi + \Delta_{d+1} Q(\xi) + \Hs_{\frac 1{2\ell}} \hat \Psi + \hat \Psi \\
& - \Big( \frac{1}{d} - Q(\xi) \Big) y\pa_y \hat \Psi + \big( 2d Q - 2 + \xi \pa_\xi Q\big) \hat \Psi + NL(\hat \Psi). 
\end{align*}
We use the expansions \eqref{eq:Psiat0} and \eqref{eq:Qat0} of $\Psi$ and $Q$, the cancellations $\big(\Hs_{\frac{1}{2\ell}} + \textup{Id}\big) \varphi_{2\ell} = 0$ and $ \big(- \frac{1}{2\ell} y\pa_y + \textup{Id}\big) y^{2\ell} = 0 $ to write for $y \leq s^\frac{1}{2\ell}$ ($\xi \leq 1$), 
\begin{align*}
 -\pa_s \Psi + &\Delta_{d+2} Q(\xi) + \Hs_{\frac 1{2\ell}} \hat \Psi + \hat \Psi \\
& \qquad =  \frac{1}{s^2} \left[ - \frac{\varphi_{2\ell}}{B_\ell} + \frac{\ell d}{B_\ell^2(2\ell + d)^2} \Delta_{d + 2} (2\alpha y^2)^{2\ell}\right]  + \Oc \Big( \frac{\yj^{6\ell - 2}}{s^3} \Big).
\end{align*}
We notice that
$$\tilde \varphi_{2\ell}(y) =\varphi_{2\ell}(y) - \frac{(2\alpha y^2)^\ell}{2\ell + d}  = \Oc(\yj^{2\ell - 2}), \quad \hat \Psi(y,s) = -\frac{1}{B_\ell s} \tilde \varphi_{2\ell} \chi_{_0}(\xi) = \Oc\Big( \frac{\yj^{2\ell - 2}}{s}\Big),$$
and use again \eqref{eq:Qat0} to expand (keep track the terms of order $\Oc(s^{-2})$ only), 
\begin{align*}
&- \Big( \frac{1}{d} - Q(\xi) \Big) y\pa_y \hat \Psi =  \frac{1}{s^2}  \frac{(2\alpha y^2)^\ell }{B_\ell^2(2\ell + d)} y\pa_y \tilde \varphi_{2\ell}   + \Oc \Big( \frac{\yj^{6\ell - 2}}{s^3} \Big), \\
&\big( 2d Q - 2 + \xi \pa_\xi Q\big) \hat \Psi  = \frac{1}{s^2} \frac{(2d + 2\ell) (2\alpha y^2)^\ell}{B_\ell^2 (2\ell + d)} \tilde \varphi_{2\ell}   + \Oc \Big( \frac{\yj^{6\ell - 2}}{s^3} \Big),\\
& NL(\hat \Psi)  = \frac{1}{B_\ell^2 s^2} \left[ d  \tilde \varphi_{2\ell}^2  + \frac{1}{2}y\pa_y  \tilde \varphi_{2\ell}^2 \right] +  \Oc \Big( \frac{\yj^{6\ell - 2}}{s^3} \Big).
\end{align*}
A collection of these expansions yields 
\begin{align}\label{dec:Ehattmp}
\hat E(y,s) = \frac{1}{B_\ell^2 s^2}P_{4\ell - 2}(y) +  \hat{R}(y,s) \quad \textup{with} \quad \sum_{j=0}^2 | (\yj \pa_y)^j \hat R(y,s)| = \Oc \Big( \frac{\yj^{6\ell - 2}}{s^3} \Big),
\end{align}
where 
\begin{align}
P_{4\ell - 2}(y) &= -B_\ell \varphi_{2\ell} + \frac{\ell d}{(2\ell + d)^2} \Delta_{d + 2} (2\alpha y^2)^{2\ell} + \frac{(2\alpha y^2)^\ell }{(2\ell + d)} y\pa_y  \tilde \varphi_{2\ell} \nonumber \\
& \qquad +  \frac{(2d + 2\ell) (2\alpha y^2)^\ell}{(2\ell + d)} \tilde \varphi_{2\ell} +  d  \tilde \varphi_{2\ell}^2  + \frac{1}{2}y\pa_y  \tilde \varphi_{2\ell}^2. \label{def:P4ell2}
\end{align}
A projection of $\hat E$ onto $\varphi_{2k}$ immediately gives 
$$ \hat E_k(s) = \langle \hat E, \varphi_{2k} \rangle_\rho = \Oc(s^{-2}), \quad k \in \mathbb{N}.$$
We claim that the projection of $P_{4\ell - 2}$ onto $\varphi_{2\ell}$ is identically zero,  i.e.
\begin{equation}\label{est:ProjectionP4ell2}
\langle P_{4\ell - 2}, \varphi_{2\ell} \rangle_\rho = 0,
\end{equation}
from which we get the improved bound 
$$\hat E_\ell(s) = \langle \hat E, \varphi_{2\ell} \rangle_\rho = \langle P_{4\ell - 2}, \varphi_{2\ell} \rangle_\rho + \langle \hat R, \varphi_{2\ell} \rangle_\rho = \Oc(s^{-3}).$$
This concludes the proofs of \eqref{dec:Ehat}, \eqref{est:Ehatkandell} and \eqref{est:Rhatpointwise}. The estimate \eqref{est:Ehatmid} is straightforward from \eqref{dec:Ehattmp} and \eqref{est:ELinf}. Indeed, we have the estimate for $j = 0,1,2$, 
\begin{align*}
\int_{1}^\infty \frac{| (y \pa_y)^j \hat E |^2}{y^{4\ell + 2}} \frac{dy}{y} &\lesssim  \int_1^{s^\frac{1}{2\ell}} \Big( \frac{| (y \pa_y)^j P_{4\ell - 2}|^2}{s^4 y^{4\ell + 2}} + \frac{| (y \pa_y)^j \hat R|^2}{ y^{4\ell + 2}}\Big) \frac{dy}{y} + \int_{s^\frac{1}{2\ell}}^\infty \frac{|(y \pa_y)^j \hat E|^2}{y^{4\ell +2}} \frac{dy}{y}\\
& \lesssim \int_1^{s^\frac{1}{2\ell}} \Big(\frac{y^{4\ell - 7}}{s^4} +\frac{y^{6\ell - 7}}{s^6} \Big)  dy + s^{-\frac{2}{\ell}}\int_{s^\frac{1}{2\ell}}^\infty \frac{dy}{y^{4\ell +3}} \lesssim s^{-2 - \frac{3}{\ell}}.
\end{align*}
It remains to prove \eqref{est:ProjectionP4ell2} to complete the proof of Lemma \ref{lemm:EEhatys}. \end{proof}

\begin{proof}[Proof of \eqref{est:ProjectionP4ell2}] We use the exact value $(d, \ell)$ to simplify the computation which we can easily implement  with Matlab symbolic. Recall from \eqref{def:phi2ell} and the relation 
$$\varphi_{2\ell}(y) = \frac{1}{y^d}\int_0^y \phi_{2\ell}(\zeta) \zeta^{d-1} d\zeta,$$ we have for $(d, \ell) = (3,3)$:
\begin{align*}
\varphi_6(y) &= \frac{y^6}{243} - \frac{2}{3}y^4 + 28y^2 - 280, \quad \tilde{\varphi}_6(y) = - \frac{2}{3}y^4 + 28y^2 - 280,\\
P_{10}(y) &= -B_3 \varphi_6(y) -\frac{4\,y^{10} }{243}+\frac{1148\,y^8 }{243}-\frac{19264\,y^6 }{81}+\frac{17360\,y^4 }{3}-62720\,y^2 +235200,
\end{align*}
and recall from \eqref{def:valueB3} that $B_3 = 39360$ to get
$$ \frac{1}{\| \varphi_{6}\|^2_\rho} \langle P_{10}, \varphi_{6} \rangle_\rho  = -B_3 + 39360 = 0. $$
Similarly, we have for $(d, \ell) = (4,2)$:
\begin{align*}
\varphi_4(y) &= \frac{y^4}{32} - 2 y^2 + 24, \quad \tilde{\varphi}_4(y) = - 2 y^2 + 24,\\
P_{6}(y) &= -B_2 \varphi_4(y)  - \frac{y^6}{8} + 33y^4 - 480y^2 + 2304,
\end{align*}
and recall from \eqref{def:valueB2} that $B_2 = 576$ to get
$$ \frac{1}{\| \varphi_{4}\|^2_\rho} \langle P_{6}, \varphi_{4} \rangle_\rho  = -B_2 + 576 = 0. $$
This concludes the proof of \eqref{est:ProjectionP4ell2} and completes the proof of Lemma \ref{lemm:EEhatys}. 
\end{proof}

\subsection{Dynamics of the finite dimensional part}
We obtain in this section the ODE satisfied by the finite dimensional part $\hat \vep_\natural$. We claim the following.
\begin{lemma}[Dynamics of the finite dimensional part $\hat \vep_\natural$] \label{lemm:finitepart} Let $\vep(s) \in \Sc_A(s)$, we have \\
(i) \textup{(Decomposition of $\tilde E$)} The function $\tilde{E}$ defined by \eqref{def:Etil} can be decomposed as 
\begin{equation}\label{est:decompEys}
\tilde E(y,s) = \sum_{k = 0, k\ne \ell}^{2\ell-1} \Big[ \hat E_k -\hat \vep_k' +  \big(1 - \frac{k}{\ell}\big)\hat \vep_k \Big] \varphi_{2k}(y) + \big(\hat E_\ell - \hat \vep_\ell' - \frac{2}{s} \hat \vep_\ell\big) \varphi_{2\ell}+  \tilde R(y,s),
\end{equation}
where $\hat E_k$ is introduced in \eqref{dec:Ehat} and satisfies the estimate \eqref{est:Ehatkandell}, 
\begin{equation}\label{est:RtilL2rho}
  \|\tilde R(s)\|_{L^2_\rho} \lesssim s^{-3}. 
 \end{equation}
(ii) \textup{(Dynamics of $\hat \vep_\natural$)}  
\begin{align} \label{eq:ODEvepk}
\sum_{k = 0, k \ne \ell}^{2\ell - 1}\left|\hat \vep_k' + \big(1 - \frac{k}{\ell}\big)\hat \vep_k \right| \lesssim s^{-2}, \quad \left| \hat \vep_\ell' + \frac{2}{s} \hat \vep_\ell \right| \lesssim s^{-3}. 
\end{align}
 
\end{lemma}
\begin{proof} (i) From the definitions \eqref{def:Etil} and \eqref{def:Ehat} of $\tilde{E}$ and $\hat E$, the decomposition \eqref{dec:Ehat} and $\Hs_{\frac{1}{2\ell}} \varphi_{2k} = -\frac{k}{\ell} \varphi_{2k}$, we have by 
\begin{align*}
\tilde{E}(y,s) &= \hat E(y,s) - \pa_s \hat \vep_\natural + (\Hs_\frac{1}{2\ell} + \textup{Id}) \hat \vep_\natural - \tilde P_1 y \pa_y \hat \vep_\natural + \tilde P_2 \hat \vep_\natural + NL(\hat \vep_\natural)\\
& = \sum_{k = 0}^{2\ell - 1}\Big[ \hat E_k - \hat \vep_k' + \big(1 - \frac{k}{\ell} \big) \Big] \varphi_{2k} + \hat R  - \tilde P_1 y \pa_y \hat \vep_\natural + \tilde P_2 \hat \vep_\natural + NL(\hat \vep_\natural),
\end{align*} 
where $NL(\hat \vep_\natural) = d \hat \vep_\natural^2 + y \hat \vep_\natural\pa_y \hat \vep_\natural$,  $\tilde{P}_1$ and $\tilde{P}_2$ are defined by 
$$\tilde{P}_1 = \frac{1}{d} - \Psi, \quad   \tilde{P}_2 =  2d \Psi - 2 + \xi \pa_\xi \Psi.$$
From the expansion \eqref{eq:Psiat0} of $\Psi$, we have the rough estimate 
$$ \forall y \geq 0, \quad |\tilde{P}_1(y,s)| + |y \pa_y \tilde{P}_1(y,s)| + |\tilde{P}_2(y,s)| \lesssim \frac{\yj^{2\ell}}{s}.$$ 
From the bootstrap bounds \eqref{est:vepk} and \eqref{est:vepell}, we have
$$ \forall y \geq 0, \quad  |\hat \vep_\natural(y,s) - \hat \vep_\ell \varphi_{2\ell}| \lesssim \frac{1}{s^2} \yj^{4\ell - 2},  \quad |\hat \vep_\natural(y,s)| \lesssim \frac{\log s}{s^2} \yj^{4\ell - 2}.$$
Using these estimates, \eqref{est:Ehatkandell} and Cauchy-Schwarz inequality, we end up with 
\begin{equation}
\|\hat R  - \tilde P_1 y \pa_y (\hat \vep_\natural(y,s) - \hat \vep_\ell \varphi_{2\ell}) + \tilde P_2 (\hat \vep_\natural(y,s) - \hat \vep_\ell \varphi_{2\ell}) + NL(\hat \vep_\natural)\|_{L^2_\rho} \lesssim s^{-3}.
\end{equation}
We claim the following:
\begin{equation}\label{est:tmpP12onvarp2ell}
 \langle - \tilde P_1 y \pa_y \varphi_{2\ell} + \tilde{P}_2 \varphi_{2\ell}, \varphi_{2\ell} \rangle_\rho \; \hat \vep_\ell(s)= -\frac{2}{s} \hat \vep_\ell(s) + \Oc(s^{-4}\log s ).
\end{equation}
We recall from \eqref{eq:Psiat0} the expansion $\Psi(y,s) = \frac{1}{d} - \frac{\varphi_{2\ell}}{B_{\ell}s} + \Oc(s^{-2} \yj^{2\ell})$, and write (keep track only terms of order $\Oc(s^{-1})$)
\begin{align*}
- \tilde P_1 y \pa_y \varphi_{2\ell} + \tilde{P}_2 \varphi_{2\ell} = -\frac{2}{B_\ell s} \Big(d\varphi_{2\ell}^2 +  y \varphi_{2\ell}\pa_y \varphi_{2\ell} \Big) + \Oc(s^{-2} \yj^{6\ell}).
\end{align*}
A direct computation (Matlab symbolic) yields 
$$ \frac{1}{\| \varphi_{2\ell}\|^2_\rho} \langle d\varphi_{2\ell}^2 +  y \varphi_{2\ell}\pa_y \varphi_{2\ell}, \varphi_{2\ell} \rangle_\rho = \left\{ \begin{array}{ll} 39360 \; &\textup{if} \; (d, \ell) = (3,3)\\
576 \; &\textup{if} \; (d, \ell) = (4,2)
\end{array}   \right.  \equiv B_\ell,$$
which agrees with the formal computation given at page \pageref{sec:computeBell} where the constant $B_\ell$ is the projection of the nonlinear term (in the original setting) onto the eigenmode $\phi_{2\ell}$. 
This proves \eqref{est:tmpP12onvarp2ell} and concludes the proof of \eqref{est:decompEys}. \\

\noindent (ii) We project the equation \eqref{eq:veptil} onto the eigenmode $\varphi_{2k}$ and use the orthogonality \eqref{eq:orhtogonalvep} to get 
\begin{equation*}
0 = \langle -\tilde V_1 y\pa_y \tilde{ \vep} + \tilde V_2 \tilde{\vep} + NL(\tilde \vep) + \tilde{E}, \varphi_{2k} \rangle_\rho. 
\end{equation*}
From \eqref{def:V12til}, \eqref{eq:Psiat0} and the bootstrap bounds \eqref{est:vepk}, \eqref{est:vepell}, we have the rough bound 
$$\forall y \geq 0, \quad |\tilde{V}_1(y,s)| + |y\pa_y \tilde{V}_1(y,s)| + |\tilde{V}_2(y,s)| \lesssim  \frac{\yj^{4\ell - 2}}{s}.$$
We use Cauchy-Schwarz inequality, integration by parts and the fact that $\rho$ is exponential decay to estimate
$$\big|\langle -\tilde V_1 y\pa_y \tilde{ \vep} + \tilde V_2 \tilde{\vep}, \varphi_{2k}\rangle_\rho \big| \lesssim s^{-1}\| \tilde{\epsilon}(s)\|_{L^2_\rho}. $$
For the nonlinear term, we use the relation $\tilde{\vep} = \hat \vep - \hat \vep_\natural$, the pointwise estimate \eqref{est:vephatpointswise} and the bootstrap bounds \eqref{est:vepk}, \eqref{est:vepell} to get 
$$\forall y \geq 0, \quad |\tilde{\vep}(y,s)|  + |y \pa_y \tilde{\vep}(y,s)| \lesssim A^6 s^{-1 - \frac{3}{2\ell}} \yj^{4\ell - 2}. $$
Then, using  Cauchy-Schwarz inequality and the exponential decay of $\rho$  yields 
$$\big| \langle  NL(\tilde \vep), \varphi_{2k} \rangle_\rho \big| \lesssim A^6 s^{-1 - \frac{3}{2\ell}} \| \tilde{\vep}(s)\|_{L^2_\rho}. $$
Putting all these estimates together with \eqref{est:decompEys} and the bootstrap bound \eqref{est:vepL2rho} yield \eqref{eq:ODEvepk} and completes the proof of Lemma \ref{lemm:finitepart}.
\end{proof}

\subsection{$L^2_\rho$-estimate}
We give the formulation to control $L^2_\rho$ of $\tilde{\vep}$. The orthogonality \eqref{eq:orhtogonalvep}, which provides the spectral gap \eqref{est:spectralgap}, plays a crucial role in the improvement of $L^2_\rho$ bootstrap estimate \eqref{est:vepL2rho}. We claim the following. 
\begin{lemma}[Energy estimate in $L^2_\rho$] \label{lemm:L2rho} Let $A \geq 1$ and $s \geq s_0 = s_0(A) \gg 1$ and  $\vep(s) \in \Sc_A(s)$, there is $\delta > 0$ such that
\begin{equation}
\frac{d}{ds}\|\tilde \vep\|_{L^2_\rho}^2 \leq - \frac{1}{2} \|\tilde \vep\|_{L^2_\rho}^2 + Cs^{-6},
\end{equation}
where $C$ is independent of $A$. 
\end{lemma}
\begin{proof} The proof is just a standard energy estimate in $L^2_\rho$ from the equation \eqref{eq:veptil}. We take the scalar product of \eqref{eq:veptil} with $\tilde{\vep}$ in $L^2_\rho$ and use the spectral gap \eqref{est:spectralgap} to get 
\begin{align*}
\frac{1}{2}\frac{d}{ds}\|\tilde \vep\|_{L^2_\rho}^2 \leq -\|\tilde \vep\|_{L^2_\rho}^2 + \Big| \langle \tilde{\Vc}\tilde{\vep} + NL(\tilde{\vep})+ \tilde E, \tilde{\vep} \rangle_\rho\big|. 
\end{align*} 
From the definition \eqref{def:Vcqtil} of $\tilde \Vc$ and integration by part, we get 
\begin{align*}
 \Big|\langle \tilde{\Vc}\tilde{\vep}, \tilde{\vep}\rangle_\rho\Big| &= \Big|-\frac{1}{2}\int_0^\infty \tilde{V_1}y \pa_y \tilde{\vep}^2 \rho dy + \int_0^\infty \tilde{V}_2 \tilde{\vep}^2 \rho dy\Big| \lesssim  \int_0^\infty \big(|y\pa_y \tilde V_1| + \yj^2 |\tilde V_1| + |\tilde V_2| \big) \tilde{\vep}^2 \rho dy. 
\end{align*}
From the definition \eqref{def:V12til} of $\tilde{V}_1$ and $\tilde{V}_2$, we have the rough bound 
\begin{equation} \label{est:V12tilpointwise}
 \forall y \geq 0, \quad  |y\pa_y \tilde V_1| + \yj^2 |\tilde V_1| + |\tilde V_2| \lesssim s^{-1} \yj^C,
\end{equation}
for some constant $C > 0$. Let $ 0< \kappa \ll  1$ be a small constant such that 
 $$\forall y \leq s^\kappa, \quad s^{-1} \yj^C \leq s^{-\kappa}. $$
We also get from \eqref{est:vephatpointswise} and the definition \eqref{def:vepnatural} of $\hat \vep_\natural$ to have the pointwise bound 
\begin{equation}\label{est:veptilpointwise}
\forall y \geq 0, \quad |\tilde \vep (y,s)| \lesssim |\hat \vep(y,s)| + |\hat \vep_\natural(y,s)| \lesssim A^3 s^{-1 - \frac{3}{2\ell}} \yj^C.
\end{equation}
By splitting the integral and using \eqref{est:V12tilpointwise}, \eqref{est:veptilpointwise}, we obtain 
\begin{align*}
\Big|\langle \tilde{\Vc}\tilde{\vep}, \tilde{\vep}\rangle_\rho\Big| &\lesssim s^{-1}\int_0^{s^\kappa} \yj^C \tilde{\vep}^2 \rho dy + A^{6}s^{-3 - \frac{3}{\ell}}\int_{s^\kappa}^\infty  \yj^C e^{-\frac{|y|^2}{2\ell}} dy\\
& \lesssim s^{-\kappa} \|\tilde{\vep}\|^2_{L^2_\rho} + A^{6}e^{-\eta s^{2\kappa}},
\end{align*}
for some $\eta > 0$. Arguing in a similar way to estimate the nonlinear term by using \eqref{est:veptilpointwise} and integration by parts, we obtain 
\begin{align*}
\big| \langle NL(\tilde\vep), \tilde{\vep}\rangle_\rho\big| &\lesssim \Big| \int_0^\infty d \tilde{\vep}^3 \rho dy +  \frac{1}{3}\int_0^\infty y\pa_y \tilde{\vep}^3 \rho dy \Big| \lesssim \int_0^\infty \yj^2 |\tilde{\vep}|^3 \rho dy\\
&\lesssim A^3 s^{-1 - \frac{3}{2\ell}} \int_0^{s^\kappa} \yj^C  |\tilde \vep|^2 \rho dy + A^{9} s^{-3 - \frac{9}{2\ell}}\int_{s^\kappa}^\infty \yj^{3C + d+1} e^{ -\frac{|y|^2}{2\ell}}dy\\
&\lesssim A^3s^{-\kappa} \|\tilde{\vep}\|^2_{L^2_\rho} + A^{9}e^{-\eta s^{2\kappa}}. 
\end{align*}
For the error term, we use the decomposition \eqref{est:decompEys}, the orthogonality \eqref{eq:orhtogonalvep}, Cauchy-Schwarz inequality and the estimate \eqref{est:RtilL2rho} to obtain 
\begin{align*}
\big| \langle \tilde{E}, \tilde \vep \rangle_\rho \big| = \big| \langle \tilde{R}, \tilde \vep \rangle_\rho \big| \leq \frac{1}{4} \|\tilde{\vep}\|^2_{L^2_\rho} + C\|\tilde{R}\|^2_{L^2_\rho} \leq \frac{1}{4} \|\tilde{\vep}\|^2_{L^2_\rho} + Cs^{-6}. 
\end{align*}
Putting together all the estimates and take $s_0 = s_0(A) \gg 1$ yields the desired formulation and concludes the proof of Lemma \ref{lemm:L2rho}.
\end{proof}

\subsection{Estimate for the intermediate region}
We perform an energy estimate to control the solution in the intermediate region $y \lesssim s^\frac{1}{2\ell} (\xi \lesssim 1)$. We claim the following. 
\begin{lemma}[Energy estimate in the intermediate region] \label{lemm:mid} Let $A \geq 1$ and $s \geq s_0 = s_0(A) \gg 1$ and $\vep(s) \in \Sc_A(s)$. We have
\begin{align}
\frac{d}{ds} \|\hat \vep(s)\|^2_{\flat}& \leq -\delta \|\hat \vep(s)\|^2_{\flat} +   C s^{-2 - \frac{3}{\ell}}, \label{est:vepmid0}\\
\frac{d}{ds} \|y\pa_y\hat \vep(s)\|^2_{\flat} &\leq -\delta \|y\pa_y\hat \vep(s)\|^2_{\flat} +   C \big(\|\hat \vep(s)\|^2_{\flat} +   s^{-2 - \frac{3}{\ell}}\big),\label{est:vepmid1}\\
\frac{d}{ds} \|(y\pa_y)^2\hat \vep(s)\|^2_{\flat} &\leq -\delta \|(y\pa_y)^2\hat \vep(s)\|^2_{\flat} +   C \big( \|y\pa_y\hat \vep(s)\|^2_{\flat} + \|\hat \vep(s)\|^2_{\flat} +   s^{-2 - \frac{3}{\ell}}\big), \label{est:vepmid2}
\end{align}
where $\delta > 0$ and $C = C(K) > 0$ is independent of $A$. 
\end{lemma} 
\begin{proof} We begin with \eqref{est:vepmid0}. To ease the notation, we write in this proof 
$$\chi_{_K}(y) = \chi(y), \quad \int_0^\infty \square dy = \int \square dy.$$
From the equation \eqref{eq:vephat}, we have
\begin{align*}
\frac{1}{2} \frac{d}{ds} \|\hat \vep(s)\|^2_{\flat} &= \int \big(1 - \chi\big) \big(\Hs \hat \vep + \hat \Vc \hat \vep + NL(\hat \vep) + \hat E \big) \frac{\hat \vep}{y^{4\ell + 3}} dy.
\end{align*}
We rewrite from the definition \eqref{def:Hs} of $\Hs$ and $\Psi = Q + \hat \Psi$, 
$$\Hs + \hat \Vc  = \Delta_{d+2} + \Big( \Psi - \frac{1}{2} \Big)y\pa_y + \Big( 2d \Psi - 1 + y\pa_y \Psi \Big).$$
We compute by integration by parts and use the fact that the compact support of $\pa_y \chi$ and $\pa_y^2 \chi$ is in $(K, 2K)$,  
\begin{align*}
\int(1 - \chi)  \frac{ \hat \vep\Delta_{d+2} \hat \vep}{y^{4\ell + 3}} dy &\leq - \int (1 - \chi) \frac{|\pa_y \hat \vep|^2}{y^{4\ell + 3}} dy + C\int \hat\vep^2 \Big( \frac{|\pa_y \chi|}{y^{4\ell + 4}} + \frac{|\pa_y^2 \chi|}{y^{4\ell +3}} \Big)dy + C \int \frac{\hat \vep^2(1 - \chi)}{y^{4\ell + 5}} dy\\
&\leq - \int (1 - \chi) \frac{|\pa_y \hat \vep|^2}{y^{4\ell + 3}} dy + \frac{C}{K^{4\ell + 3}} \int_K^{2K} |\hat \vep|^2 dy + \frac{C}{K^2}\| \hat\vep\|^2_\flat. 
\end{align*}
Using the relation $\hat \vep = \tilde{\vep} + \hat \vep_\natural$ and the bootstrap bounds in Definition \ref{Definition-shrinking -set}, we obtain 
\begin{align}
\int_K^{2K} |\hat \vep|^2 dy &\leq \frac{e^{\frac{K^2}{\ell}}}{ K^{d+1}} \|\hat \vep\|_{L^2_\rho}^2 \leq \frac{e^{\frac{K^2}{\ell}}}{ K^{d+1}} \big( \|\tilde \vep\|_{L^2_\rho}^2 + \| \hat \vep_\natural\|^2_{L^2_\rho} \big) \nonumber \\
& \qquad \leq \frac{e^{\frac{K^2}{\ell}}}{ K^{d+1}}\Big(\frac{A^{6}}{ s^{6}} + \frac{A^4 \log^2s}{ s^{4}}\Big) \leq s^{-2 - \frac{3}{\ell}}. \label{est:vephapK}
\end{align}
Using integration by parts , we derive 
\begin{align*}
&\int \frac{(1 - \chi)\hat \vep}{y^{4\ell + 3}} \left[\Big( \Psi - \frac{1}{2} \Big)y\pa_y \hat \vep + \Big( 2d \Psi - 1 + y\pa_y \Psi \Big) \hat \vep\right]dy\\
& \quad  =-\int   \frac{(1 - \chi)\hat \vep^2}{y^{4\ell + 3}}\Big[(2\ell +1) \big( \frac 12 - \Psi \big) - \frac{1}{2} y\pa_y\Psi + 1 - 2d \Psi\Big] dy + \int \frac{\hat \vep^2 \pa_y \chi}{2 y^{4\ell +2}} dy. 
\end{align*}
We now use the monotonicity of $Q$ stated in \eqref{est:propQ} and the fact that $\|\hat \Psi(s)\|_\infty = \Oc(s^{-\frac{1}{2\ell}})$ to have
$$\forall y \geq 0, \quad  \frac{1}{2} - \Psi(y,s) \geq \frac{1}{2\ell} - Cs^{-\frac{1}{2\ell}}, \quad \Psi(y,s) \leq \frac{1}{d} + Cs^{-\frac{1}{2\ell}}, \quad y\pa_y \Psi(y,s) \leq Cs^{-\frac{1}{2\ell}},$$
hence, 
$$\forall y \geq 0, \quad  (2\ell +1) \big( \frac 12 - \Psi \big) - \frac{1}{2} y\pa_y\Psi + 1 - 2d \Psi \geq \frac{2\ell + 1}{2\ell} - 1 - Cs^{-\frac{1}{2\ell}} = \frac{1}{2\ell} - Cs^{-\frac{1}{2\ell}} \geq  \frac{1}{4\ell}. $$
The term with cutoff $\pa_y \chi$ is simply estimated as in \eqref{est:vephapK},  
$$ \Big|\int \frac{\hat \vep^2 \pa_y \chi}{2 y^{4\ell +2}} dy \Big| \lesssim K^{-4\ell -2}\int_K^{2K} |\hat \vep|^2 dy \lesssim s^{-2 - \frac{3}{\ell}}.$$
Hence, by taking $K \gg 1$ large,  the full linear term is estimate by 
\begin{equation}\label{est:linearestimate}
\int (1 - \chi) \frac{\hat \vep(\Hs + \hat \Vc) \hat \vep}{y^{4\ell +3}} dy \leq  - \int (1 - \chi) \frac{|\pa_y \hat \vep|^2}{y^{4\ell + 3}} dy  -\frac{1}{6\ell} \| \hat \vep(s)\|^2_\flat + Cs^{-2 - \frac{3}{\ell}}. 
\end{equation}
As for the nonlinear term, we estimate by using \eqref{est:vepLinfbound}, 
\begin{align*}
\Big| \int (1 - \chi) \frac{\hat \vep NL(\hat \vep)}{y^{4\ell + 3}}dy \Big| &= \Big| \int (1 - \chi) \frac{\hat \vep^2 (d\hat \vep + y \pa_y \hat \vep)}{y^{4\ell + 3}}dy \Big|\\
& \qquad \leq (\|\hat \vep(s)\|_\infty + \| y\pa_y \hat \vep(s)\|_\infty) \|\hat \vep(s)\|^2_\flat \lesssim A^8s^{-\frac{1}{\ell}}   \|\hat \vep(s)\|^2_\flat. 
\end{align*}
For the error term, we use Cauchy-Schwarz inequality and \eqref{est:Ehatmid}, 
\begin{align*}
\int (1 - \chi) \frac{|\hat \vep| |\hat E|}{y^{4\ell +3}} dy \leq \frac{1}{8\ell} \|\hat \vep(s)\|^2_\flat + C\| \hat E\|^2_\flat \leq \frac{1}{8\ell} \|\hat \vep(s)\|^2_\flat + Cs^{-2 - \frac{3}{\ell}}. 
\end{align*}
Collecting all the above bounds and taking $K \gg 1$ and $s_0(A) \gg 1$, we end up with
\begin{align*}
 \frac{1}{2} \frac{d}{ds}\|\hat \vep(s)\|^2_\flat &\leq \Big( - \frac{1}{6\ell} + \frac{1}{8\ell} + \frac{CA^8}{s^{\frac{1}{\ell}}}\Big) \|\hat \vep(s)\|^2_\flat   + Cs^{-2 - \frac{3}{\ell}} \leq -  \delta \|\hat \vep(s)\|^2_\flat  + Cs^{-2 - \frac{3}{\ell}},
 \end{align*}
for some $\delta > 0$, which concludes the proof of \eqref{est:vepmid0}. The derivation of \eqref{est:vepmid1} and \eqref{est:vepmid2} is similar as for \eqref{est:vepmid0}. Indeed, the equations satisfied by 
$$ g_1 = y\pa_y \hat \vep, \quad g_2 = y\pa_y g_1,$$
have the same forms as for $\hat \vep$ with an extra commutator,
\begin{align*}
\pa_s g_1 &= (\Hs + \hat\Vc)g_1 + [y\pa_y, \Hs + \hat \Vc] \hat \vep + y\pa_y (NL(\hat \vep) + \hat E), \\
\pa_s g_2 &= (\Hs + \hat\Vc)g_2 + [y\pa_y, \Hs + \hat \Vc] g_1 + y\pa_y ( [\Hs + \hat \Vc, y\pa_y] \hat \vep ) +  (y\pa_y)^2(NL(\hat \vep) + \hat E).
\end{align*}
The linear part is estimated as in \eqref{est:linearestimate} that provides a dissipative term and a coercive estimate with the constant $- \frac{1}{6\ell}$. The commutator term $[y\pa_y, \Hs + \hat \Vc] \hat \vep$ is then controlled either by the dissipation or by the norm $\|\hat \vep(s)\|^2_\flat$, and similarly for the term $[y\pa_y, \Hs + \hat \Vc] g_1$. The nonlinear term and the error term are estimated by integration by parts, Cauchy-Schwarz inequality and the dissipation with the provided estimate \eqref{est:Ehatmid} of the error term that we omit the detail here. This concludes the proof of Lemma \ref{lemm:mid}.
\end{proof}

\subsection{Estimate for the outer region} \label{sec:outer}
This section is devoted to the control of the remainder $\vep$ in the outer region $y \gg s^\frac{1}{2\ell} (\xi \gg 1)$ based on the well-known semigroup properties of the Hermite operator $\Ls_\eta = \Delta - \eta z\cdot \nabla$ with $\eta = \frac{1}{2}$.   Let $\chi_{_K}$ be the cut-off function defined by \eqref{def:chiK} and recall the definition of $\vep^\out$, 
$$\vep^\out(y,s) = \vep(y,s) (1 - \chi_{_K}(\xi)), \quad \xi = ys^{-\frac{1}{2\ell}}.$$
In what follows, we write without distinguishing
$$\vep = \vep(z, s) \equiv \vep(y, s), \quad \vep^\out = \vep^\out(z, s) \equiv \vep^\out(y, s), \quad y = |z|, \;\; z \in \Rb^d,$$
and notice that
$$z\cdot\nabla_z \vep \equiv y\pa_y \vep, \quad \nabla \cdot(z \vep) = z\cdot\nabla \vep + d \vep \equiv y\pa_y \vep + d \vep,$$ 
From \eqref{eq:vepys}, we have the equation satisfied by $\vep^\out$,
\begin{equation}\label{eq:vepout1}
\pa_s \vep^\out  = \big(\Ls_\eta - \textup{Id}\big)\vep^\out+ F + \Ec^{bd}, \quad \eta = \frac{1}{2},
\end{equation} 
where 
\begin{align}
F &= \big(1 - \chi_{_{K}} \big)\big[ \big( 2y^{-2} +  Q\big) y\pa_y \vep + (2dQ + y \pa_yQ) \vep + NL(\vep) + E\big], \nonumber \\
& =  \big(1 - \chi_{_{K}} \big)\big[ P_1 y\pa_y \vep + P_2 \vep + NL(\vep) + E\big] = \big(1 - \chi_{_{K}} \big) \hat F, \label{def:F} \\
\Ec^{bd} & = \Big(-\pa_s \chi_{_K} + \Delta \chi_{_K}  - \frac 1{2\ell} y \pa_y \chi_{_K}\Big) \vep+ 2\pa_y \chi_{_K}\pa_y \vep. \label{def:EcbdK}
\end{align}
We restate some well-known semigroup properties of the Hermite operator $\Ls_\eta = \Delta - \eta z.\nabla$ acting on general functions (not necessary radially symmetric) defined from $\Rb^{d}$ to $\Rb$ 
\begin{lemma}[Properties of the semigroup $e^{s \Ls_\eta}$] \label{lemm:semigroupH} The kernel of the semigroup $e^{s \Hs_\eta}$ is given by 
\begin{equation}\label{def:kernelLalp}
e^{s\Ls_\eta}(z,\xi)  = \frac{1}{[2\pi \eta (1 - e^{-s})]^\frac{d}{2}} \exp\Big(  - \frac{\eta}{2} \frac{ |z e^{-s/2} - \xi|^2}{(1 - e^{-s})} \Big). 
\end{equation}
The action of $e^{s \Ls_\eta}$ on the function $g: \Rb^{d} \to \Rb$ is defined by 
$$e^{s \Ls_\eta} g(z) = \int_{\Rb^{d}} e^{s\Ls_\eta} (z,\xi) g(\xi ) d\xi.  $$
We have the following properties: \\
(i) $\big\| e^{s \Ls_\eta} g\big\|_\infty \leq \|g\|_\infty$ for all $g \in L^\infty(\Rb^{d})$.\\
(ii) $\big\| e^{s \Ls_\eta}\nabla g \big\|_\infty \leq \frac{C}{\sqrt{1 - e^{-s}}} \|g\|_\infty$ for all $g \in L^\infty(\Rb^{d})$.\\
\end{lemma}
\begin{proof} The formulation \eqref{def:kernelLalp} can be verified by a direct check after a simple change of variable, thanks to the fact that the function $\rho_0(z) = e^{-\frac{\eta |z|^2}{2}}$ satisfies $\Delta \rho_0 + \eta z \cdot \nabla \rho_0 + d\eta\rho_0 = 0$ for all $z \in \Rb^{d}$. The estimates in (i)-(ii) are straightforward from \eqref{def:kernelLalp}. 
\end{proof}

\begin{lemma}[Estimates in the outer region] \label{lemm:outer} For $A \geq 1$ and $s_0 = s_0(A) \gg 1$ and $\vep(s) \in \Sc_A(s)$, we have for all $ \tau \in [s_0, s]$, 
\begin{equation}\label{est:vepoutDuh}
j = 0,1, \quad  \|(y\pa_y)^j \vep^\out(s)\|_{L^\infty} \leq e^{-(s - \tau)} \|(y\pa_y)^j\vep^\out(\tau)\|_{L^\infty} + \frac{C(K) A^{3+j}}{\tau^{\frac{1}{\ell}}}(1 + s - \tau),
\end{equation}
and 
\begin{equation}\label{est:vepoutDy1}
 \|y\vep^\out(s)\|_{L^\infty} \leq e^{-(s - \tau)} \|y\vep^\out(\tau)\|_{L^\infty} + \frac{C(K) A^{3}}{\tau^{\frac{1}{2\ell}}}(1 + s - \tau),
\end{equation}
\end{lemma}
\begin{proof} We use Duhamel's formula and item (i) of Lemma \ref{lemm:semigroupH} to write from \eqref{eq:vepout1} for all $\tau \in [s_0, s]$, 
\begin{align*}
\|\vep^\out(s)\|_{L^\infty} \leq e^{-(s - \tau)}\|\vep^\out(\tau)\|_{L^\infty} + \int_\tau^s e^{-(s - s')}\big(\| F(s')\|_{L^\infty} + \|\Ec^{bd}(s')\|_{L^\infty}\big) ds',
\end{align*}
 Due to  the the cut-off $\chi_{_K}(\xi)$, the boundary term $\Ec^{bd}(\vep)$ is located in the zone $K s^{\frac{1}{2\ell}} \leq y \leq 2K s^{\frac{1}{2\ell}} (K \leq \xi \leq 2K)$  and it is bounded by using the estimate from the intermediate region. In particular, we have from \eqref{est:vephatpointswise} and the bootstrap bounds \eqref{est:vepk} and \eqref{est:vepell} the following estimate for $j = 0,1$,
$$ K \leq \xi \leq 2K, \quad | (y\pa_y)^j\vep(y,s)| \leq |(y\pa_y)^j\hat \vep(y,s)| + |(y\pa_y)^j\hat \vep_\natural(y,s)| \leq C(K)A^3 s^{-\frac{1}{\ell}}.$$
Hence, from the definition \eqref{def:EcbdK}, we obtain
$$ \| \Ec^{bd} (s)\|_{L^\infty} \lesssim \| \vep(s)\|_{L^\infty(K \leq \xi \leq 2K)} +  \|y \pa_y\vep(s)\|_{L^\infty(K \leq \xi \leq 2K)} \lesssim  C(K) A^3 s^{-\frac{1}{\ell}}. $$
For the term $F$, we use the decay of $Q$ that is 
$$\forall y \geq K s^\frac{1}{2\ell}, \quad |Q(\xi)| + |\xi \pa_\xi Q(\xi)| \lesssim \xi^{-2} \lesssim K^{-2}s^{-\frac{1}{\ell}},$$
and from the definition \eqref{def:NLq} of $NL(\vep)$ and the bootstrap bound \eqref{est:vepout} and the bound \eqref{est:ELinf}, we get 
\begin{align*}
\|F\|_{L^\infty} &\lesssim s^{-\frac{1}{\ell}} \big(\| \vep^\out\|_{L^\infty} + \| y\pa_y \vep^\out\|_{L^\infty}\big) + \|\vep^\out\|^2_{L^\infty} + \|\vep^\out\|_{L^\infty}\| y\pa_y \vep^\out\|_{L^\infty} + \|E\|_{L^\infty}\\
& \lesssim A^{9}s^{-\frac{2}{\ell}} + s^{-\frac{1}{\ell}} \lesssim s^{-\frac{1}{\ell}},
\end{align*}
for $s_0(A) \gg 1$ so that  $A^9 s_0^{-\frac{1}{\ell}} \lesssim 1$. We gather all these estimates and simply bound $\int_\tau^s e^{-(s - s')} s'^{-\frac{1}{\ell}} ds' \lesssim \tau^{-\frac{1}{\ell}} (1 + s - \tau)$ to conclude the estimate \eqref{est:vepoutDuh} for $j = 0$. \\
The proof of \eqref{est:vepout} for the case $j =1$ is similar  as for $j = 0$ by using (ii) of Lemma \ref{lemm:semigroupH}. The only difference is due to the extra commutator term in the equation satisfied by 
$$g^\out = z \cdot \nabla \vep^\out \equiv y\pa_y \vep^\out ,$$
which reads as
\begin{equation}\label{eq:gout1}
\pa_s g^\out  = \big(\Ls_\eta - \textup{Id}\big)g^\out+ [z\cdot \nabla ,\Delta_{d+2}] \vep^\out + z\cdot \nabla (F + \Ec^{bd}),
\end{equation} 
where 
$$[z\cdot \nabla , \Delta_{d+2}] \vep^\out =  - 2\Delta \vep^\out = -2\nabla \cdot \Big( \frac{z g^\out}{y^2}\Big).$$
Let 
$$g = z\cdot \nabla \vep \equiv y\pa_y \vep, $$
we write from the definition \eqref{def:F} of $F$, 
\begin{align*}
z\cdot \nabla F & = -y\pa_y \chi_{_K} (\hat F + yP_1 g ) + \pa_y(yP_1 g (1 - \chi_{_K})) \\
& \quad \qquad +  (1 - \chi_{_K})\Big[ (y\pa_y P_1 + P_2 - yP_1)g + y\pa_y P_2 \vep + E + NL(\vep)\Big] ,
\end{align*}
and from the definition \eqref{def:NLq} of $NL$, 
$$  (1 - \chi_{_K})y\pa_y NL(\vep) =  (1 - \chi_{_K})\Big[ (2d-1) \vep g + g^2\Big]  - y\pa_y\chi_{_K} \vep g + \pa_y(y\vep^\out g). $$
Hence, 
\begin{equation}
z\cdot \nabla F = \pa_y \big(yP_1 g (1 - \chi_{_K}) + y\vep^\out g) + G,
\end{equation}
where we can bound $G$ in $L^\infty$ from the bootstrap estimates \eqref{est:vepout}, the decay of $Q$, the support of $\chi_{_K}$ and its derivatives, 
\begin{align*}
\|G(s)\|_{L^\infty} \lesssim s^{-\frac{1}{\ell}}, \quad \| yP_1 g(1 - \chi_{_K}) + y \vep^\out g)\|_{L^\infty} \lesssim A^9 s^{-\frac{3}{2\ell}} \lesssim s^{-\frac{1}{\ell}}. 
\end{align*}
Similar, we have 
\begin{align*}
z\cdot \nabla \Ec^{bd} = 2\pa_y(\pa_y\chi_{_K} g) + G^{bd}, 
\end{align*}
where $G^{bd}$ and $\pa_y\chi_{_K} g$ have supports on $\{Ks^{\frac{1}{2\ell}} \leq y \leq 2Ks^{\frac{1}{2\ell}}\}$ that can be bounded using the estimate \eqref{est:vephatpointswise}, 
$$\|G^{bd}(s)\|_{L^\infty} + \|\pa_y\chi_{_K} g\|_{L^\infty} \lesssim A^3 s^{-\frac{1}{\ell}}. $$
We now use the Duhamel's formula applied to \eqref{eq:gout1}, Lemma \ref{lemm:semigroupH} and \eqref{est:ELinf}  to get 
\begin{align*}
\|g^\out(s)\|_{L^\infty} &\leq e^{-(s - \tau)}\|g^\out(\tau)\|_{L^\infty} + \int_\tau^s \frac{e^{-(s - s')}}{\sqrt{1 - e^{-(s - s')}}} \Big[ \| y^{-1}g^\out\|_{L^\infty} +   \|\pa_y\chi_{_K} g\|_{L^\infty} +   \|\pa_y\chi_{_K} g\|_{L^\infty}\Big] \\
&  \qquad \qquad  + \int_\tau^s e^{-(s - s')} \Big[ \| G(s')\|_{L^\infty} + \| G^{bd}(s')\|_{L^\infty} + \|y\pa_y E(s')\|_{L^\infty}\Big] ds'\\
& \lesssim e^{-(s - \tau)}\|g^\out(\tau)\|_{L^\infty} + A^3\int_\tau^s \frac{e^{-(s - s')}}{\sqrt{1 - e^{-(s - s')}}}  (s')^{-\frac{1}{
\ell}} ds' +  \int_\tau^s e^{-(s - s')}(s')^{-\frac{1}{
\ell}}ds'\\
& \lesssim e^{-(s - \tau)}\|g^\out(\tau)\|_{L^\infty} + A^3 \tau^{-\frac{1}{\ell}} (1 + s - \tau). 
\end{align*}
This concludes the proof of \eqref{est:vepoutDuh} for $j = 1$. The estimate \eqref{est:vepoutDy1} for $\|y \vep^\out\|_{L^\infty}$ follows the same proof as for \eqref{est:vepoutDuh}, except that the bound on the error $\| yE (1 - \chi_{_K})\|_{L^\infty} \lesssim s^{-\frac{1}{2\ell}}$. This completes the proof of Lemma \ref{lemm:outer}.  
\end{proof}

\subsection{Proof of Proposition \ref{prop:3} and  Theorem \ref{theo:1}} \label{sec:proofofmainthm}
In this section we give the proof of Proposition \ref{prop:3} to complete the proof of Proposition \ref{prop:2}. Theorem \ref{theo:1} is a direct consequence of Proposition \ref{prop:2}.

\begin{proof}[Proof of Proposition \ref{prop:3}]  The basic idea is to improve the bootstrap estimates given in Definition \ref{Definition-shrinking -set}, except for the first $\ell$ modes $(\hat \vep_k)_{0 \leq k \leq \ell-1}$. Regarding the constants, we fix them in the following order:  we fix $K \gg 1$ a large constant independent of $A$, then $A = A(K) \gg 1$, then $s_0 = s_0(A) \gg 1$. We recall from the assumption that 
$$\vep(s) \in \Sc_A(s) \quad \forall s \in [s_0, s_1] \quad \textup{and} \quad \vep(s_1) \in  \pa \Sc_A(s_1).$$
(i) (\textit{Improve bootstrap estimates}) Let's begin with $\hat \vep_\ell $ and argue by contradiction that  there is $\bar s \in [s_0, s_1]$ such that 
$$ |\hat \vep_\ell(s) | < \frac{A^2 \log s}{s^2} \; \forall s\in [s_0, \bar s), \quad  \hat \vep_\ell(\bar s) = \pm\frac{A^2 \log \bar s}{\bar s^2},$$
then, we have by equation \eqref{eq:ODEvepk} (consider that case $\hat \vep(\bar s) > 0$, similar for the negative case)
$$ -\frac{2A^2 \log \bar s}{\bar s^3} + \frac{C}{\bar s^3}  \hat \vep_\ell'(\bar s) \geq A^2 \frac{d}{ds} \frac{A^2 \log s}{s^2}\Big\vert_{\bar s} =  \frac{A^2}{\bar s^3} - \frac{2A^2 \log \bar s}{\bar s^3},$$
which can not happen for $A$ large enough. Therefore, $\hat \vep_\ell(s)$ never touches its boundary, 
$$|\hat \vep_\ell(s_1) | < \frac{A^2 \log s_1}{s_1^2}.$$
As for the modes $\hat \vep_k$ with $k = \ell + 1, \cdots, 2\ell - 1$, we integrate the ODE \eqref{eq:ODEvepk} forward in time  and use the fact that the eigenvalue is negative to conclude that $\hat \vep_k(s)$ can not touch its boundary as well. The same way for $\|\tilde \vep(s)\|_{L^2_\rho}$ and $\|(y\pa_y)^j \hat \vep\|_\flat$ thanks to the energy estimates derived in Lemmas \ref{lemm:L2rho} and \eqref{lemm:mid}. The improvement of $\| (y\pa_y)^j\vep^\out\|_{L^\infty}$ and $\|y \vep^\out\|_{L^\infty}$ follows from Lemma \ref{lemm:outer} by taking $\lambda = \log A \gg 1$ and $s_0 \geq \lambda$ such that for all $\tau \geq s_0$ and $s \in [\tau, \tau + \lambda]$, we have 

$$\tau \leq s \leq \tau +\lambda \leq \tau + s_0 \leq 2\tau,  \quad \textup{hence}, \quad \frac{1}{2\tau}\leq \frac{1}{s} \leq \frac{1}{\tau}\leq \frac{2}{s}. $$

This give us the bound 
$$\frac{C(K) A^{3+j}}{\tau^{\frac{1}{\ell}}}(1 + s - \tau) \lesssim \frac{C(K) A^{3+j} \log A}{s^{\frac{1}{\ell}}} < \frac{A^{4+j}}{s^\frac{1}{\ell}},$$
for $A$ large enough. This concludes that $\vep(s_1)$ can only touch its boundary $\pa \Sc_A(s_1)$ at the first $\ell$ modes $(\hat \vep_k)_{0 \leq k \leq \ell-1}$. \\
(ii) (\textit{Transverse crossing}) The estimate \eqref{transversecross} follows from a direct computation thanks to \eqref{eq:ODEvepk}, 
\begin{equation}
\frac{1}{2}\frac{d}{ds}\sum_{k = 0}^{\ell - 1} \hat \vep^2_k(s_1) = \sum_{k = 0}^{\ell - 1} \Big[ (1 - k/\ell) \hat \vep^2(s_1) + \Oc(s^{-2} |\hat \vep(s_1)|)\Big] \geq \frac{A^4 - CA^2}{s_1^4} > 0,
\end{equation}
for $A$ large enough. This completes the proof of Proposition \ref{prop:3} as well as Proposition \ref{prop:2}. 
\end{proof}

\begin{proof}[Proof of Theorem \ref{theo:1}]  (i) and (ii) follows from the definition of the shrinking set \ref{Definition-shrinking -set} and the relation $w = d v + y\pa_y v$ and $\phi_{2\ell} = d \varphi_{2\ell} + y\pa_y \varphi_{2\ell}$. As for (iii), we use the same argument as in Herrero-Vel\'azquez \cite{HVaihn93} for the classical nonlinear heat equation (see also Bebernes-Bricher \cite{BBsima92}, Zaag \cite{ZAAihn98}, \cite{GNZana20} for a similar approach), we only sketch the computation for the reader convenience. We introduce the  auxiliary function
$$g(x_0, \xi, \tau) = (T - t_0)u(x,t), \quad x = x_0 + \xi \sqrt{T- t_0}, \quad t = t_0 + \tau(T-t_0),$$
where $t_0 = t_0(x_0)$ is uniquely determined by 
$$|x_0| = K_0\sqrt{T - t_0} | \log(T - t_0)| ^\frac{1}{2\ell}, \quad K_0 \gg 1.$$
We have the relation 
$$\log (T - t_0) \sim 2\log |x_0|, \quad T - t_0 \sim \frac{|x_0|^2}{K_0^2 \big(2 |\log |x_0|| \big)^ \frac{1}{\ell} }. $$
From \eqref{exp:innerIntro}, we have 
\begin{align*}
u^*(x_0) = \lim_{t \to T} u(x,t) = (T - t_0)^{-1} \lim_{\tau \to 1} g(x_0,0, \tau) = (T-t_0)^{-1}\hat g_{K_0}(1).
\end{align*}
We compute from \eqref{est:QxiInf},
$$\hat g_{K_0}(1) = F(K_0) \sim (d - 2) c_\ell^{-\frac{1}{\ell}} K_0^{-2}, $$
which gives
$$u^*(x_0) \sim  (d- 2) \left(\frac{2}{c_\ell} \right)^\frac{1}{\ell} \frac{|\log |x_0|| ^ \frac{1}{\ell}}{|x_0|^2} \quad \textup{as}\;\; |x_0| \to 0.$$
This completes the proof of Theorem \ref{theo:1}.
\end{proof}

\def\cprime{$'$}

\end{document}